\documentclass[a4paper]{amsart}

\usepackage[a4paper,width={16cm},left=2.5cm,bottom=3cm, top=3cm]{geometry}

\usepackage[utf8]{inputenc}
\usepackage{cancel}
\usepackage{soul}
\usepackage{latexsym}
\usepackage{a4wide}
\usepackage{amscd}
\usepackage{graphics}
\usepackage{amsmath}
\usepackage{amssymb}
\usepackage{bm}
\usepackage{mathrsfs}
\usepackage{enumerate}
\usepackage{csquotes}
\usepackage{pgf,tikz}

\usepackage{hyperref}
\usepackage{cleveref}
\hypersetup{
	linktocpage=true,
	colorlinks=true,
	linkcolor=blue!70!black,
	bookmarks=true,
	citecolor=orange!40!red,
	urlcolor=magenta!80!black,
}
\usepackage{todonotes}

\numberwithin{equation}{section} 

\newtheorem{theorem}{Theorem}[section]
\newtheorem{corollary}[theorem]{Corollary}
\newtheorem{lemma}[theorem]{Lemma}
\newtheorem{proposition}[theorem]{Proposition}
 \theoremstyle{definition} 
 \newtheorem{definition}[theorem]{Definition}

 \newtheorem{remark}[theorem]{Remark}

\newcommand{\R}{\mathbb{R}}	
\newcommand{\N}{\mathbb{N}} 
\newcommand{\dx}{\,\mathrm{d}x}	
\newcommand{\ds}{\,\mathrm{d}S}	
\renewcommand{\d}{\mathrm{d}}
\newcommand{\e}{\varepsilon}	
\newcommand{\nnu}{\bm{\nu}}  

\newcommand{\norm}[1]{\left\lVert #1 \right\lVert}
\newcommand{\abs}[1]{\left| #1 \right|}
\newcommand{\sub}{\subseteq}
\newcommand{\weak}{\rightharpoonup}
\newcommand\restr[1]{\raisebox{-.5ex}{$|$}_{#1}}

\newenvironment{bvp}{\left\{\begin{aligned}  }{\end{aligned}\right.}

\begin{document}

\title[Sharp behavior of Dirichlet--Laplacian eigenvalues]{Sharp behavior of Dirichlet--Laplacian eigenvalues for a class of singularly perturbed problems}

\author[L. Abatangelo]{Laura Abatangelo}
\address{Laura Abatangelo
 \newline \indent Politecnico di Milano
 \newline \indent Dipartimento di Matematica
 \newline\indent Piazza Leonardo da Vinci 32, 20133 Milano, Italy}
\email{laura.abatangelo@polimi.it}

\author[R. Ognibene]{Roberto Ognibene}
\address{Roberto Ognibene
	\newline \indent Universit\`a di Pisa
	\newline \indent Dipartimento di Matematica
	\newline\indent  Largo Bruno Pontecorvo, 5, 56127 Pisa, Italy}
\email{roberto.ognibene@dm.unipi.it}

\date{\today}

\maketitle

\noindent
{{\bf Abstract.} We deepen the study of Dirichlet eigenvalues in bounded domains where a thin tube is attached to the boundary. As its section shrinks to a point, the problem is spectrally stable and we quantitatively investigate the rate of convergence of the perturbed eigenvalues. We detect the proper quantity which sharply measures the perturbation's magnitude. It is a sort of torsional rigidity of the tube's section relative to the domain. This allows us to sharply describe the asymptotic behavior of the perturbed spectrum, even when eigenvalues converge to a multiple one. The final asymptotics of eigenbranches depend on the local behavior near the junction of eigenfunctions chosen in a proper way.

The present techniques also apply when the perturbation of the Dirichlet eigenvalue problem consists in prescribing homogeneous Neumann boundary conditions on a small portion of the boundary of the domain.
}

\noindent {\bf Keywords.} Dirichlet-Laplacian; multiple eigenvalues; singular perturbations of domains; varying mixed boundary conditions; torsional rigidity

\medskip 

\noindent{\bf MSC classification.} 35J25; 
35P15; 
35B25. 

\section{Introduction and main results}

The object of the present paper is the eigenvalue variation of Dirichlet eigenvalues for a class of singularly perturbed domains. 
Our main attention is devoted to the sharp effect of attaching a thin tube to a fixed bounded domain when its cross-section shrinks to zero. This class of problems covers also the case of the widely studied \textit{dumbbell domains}. In addition, the presented proofs fit also the case of moving mixed Dirichlet--Neumann boundary conditions, as the Neumann part tends to disappear. In both cases, the starting \emph{unperturbed} problem is the classical eigenvalue problem for the Laplacian with Dirichlet boundary conditions. 

More precisely, we fix $d\geq 2$ and a bounded, open and connected set $\Omega\sub\R^d$ and we consider the problem of finding $\lambda\in\R$ and a nonzero function $\varphi\colon \Omega\to \R$ such that
\begin{equation}\label{eq:str_dir_eigen}
	\begin{bvp}
		-\Delta \varphi &=\lambda\varphi, &&\text{in }\Omega, \\
		\varphi &=0, &&\text{on }\partial\Omega.
	\end{bvp}
\end{equation}
This is considered in a weak sense: $\varphi\in H^1_0(\Omega)\setminus\{0\}$ is such that
\begin{equation*}
	\int_\Omega\nabla \varphi\cdot\nabla v\dx=\lambda\int_{\Omega} \varphi v\dx,\quad\text{for all }v\in H^1_0(\Omega).
\end{equation*}
If this happens, we say that $\lambda$ is an \emph{eigenvalue} and $\varphi$ is one of the corresponding \emph{eigenfunctions}. By classical spectral theory it is known that this problem admits a diverging sequence of positive eigenvalues, which we hereafter denote 
\[
	0<\lambda_1<\lambda_2\leq \cdots\leq \lambda_n\leq \cdots\to +\infty.
\]
We also denote by $\{\varphi_n\}_{n\geq 1}$ a corresponding sequence of eigenfunctions, assumed to be orthonormal in $L^2(\Omega)$.

We now introduce a singular perturbation of \eqref{eq:str_dir_eigen}, which basically consists in attaching a thin tube to the domain $\Omega$. Let us assume that $0\in\partial \Omega$ and $\partial \Omega $ is flat in a neighbourhood of the origin, namely
\begin{equation}\label{eq:ass_flat}
		  B_{r_0}':=\{x\in B_{r_0}\colon x_d=0\}\sub \partial \Omega~\text{and}~ B_{r_0^+}:=B_{r_0}\cap \{x_d>0\}\sub\Omega~\text{for some }r_0>0.
\end{equation}
Given a relatively open set $\Sigma\sub B_{r_0}'$ with Lipschitz boundary and $\e\in(0,1)$, we consider the thin tube of section $\e\Sigma$ and fixed height equal to $1$ attached at the origin. If we denote 
\[
 T_\e:= \e \Sigma \times (-1,0],
\]
our perturbed domain will be 
\begin{equation}\label{eq:Omega_eps}
\Omega_\e:=\Omega\cup \e\Sigma\cup T_\e 
\end{equation}
and $\overline{\Omega}\cap \overline{T_\e}=\overline{\e\Sigma}$. We then consider the eigenvalue problem for the Dirichlet-Laplacian on the perturbed domain $\Omega_\e$
\begin{equation}\label{eq:perturbed_problem}
 \begin{bvp}
		-\Delta \varphi &= \lambda^\e \varphi, &&\text{in }\Omega_\e, \\
		\varphi &=0, &&\text{on }\partial\Omega_\e. \\
	\end{bvp}
\end{equation}
Again by classical spectral theory, for $\e\in(0,1)$ this problem admits a sequence of eigenvalues tending to $+\infty$, which will be denoted as
\[
	0<\lambda_1^\e<\lambda_2^\e\leq \cdots\leq \lambda_n^\e\leq \cdots\to +\infty,
\]
whereas $\{\varphi_n^\e\}_{n\geq 1}$ will denote a corresponding sequence eigenfunctions, assumed to be orthonormal in $L^2(\Omega_\e)$.

The main goal of the present paper is understanding the behavior of the perturbed eigenvalues $\lambda_n^\e$ as the tube radius $\e$ tends to zero. Literature on this problem is very rich. We refer to \cite[Introduction]{FO} for a presentation of the established results. Here we just mention that such a large interest for the problem is due to physical and engineering motivations. Spectral behavior of the Laplacian on thin branching domains appears in the theory of quantum graphs, which models propagation of waves in quasi one-dimensional systems (quantum wires and waveguides, photonic crystals, blood vessels and so on), as well as in the theory of elasticity and multistructure problems.

First of all, as a consequence of the convergence in the sense of Mosco of the sequence of domains $\{\Omega_\e\}_\e$ to the limit domain $\Omega$ (see the discussion in Subsection \ref{subsec:mosco}), classical results (see e.g. \cite{daners2003}) ensure stability of the spectrum, in the sense that for any $N\in\N\setminus\{0\}$
\[
 \lambda_N^\e \to \lambda_N \quad \text{as }\e\to 0.
\]
The analysis of the present paper originates from the main result in \cite{FO}, which in turn generalizes the papers \cite{AFT2014} and \cite{Gadylshin2005tubi}. In \cite{FO} the authors study the sharp asymptotic behavior of Dirichlet eigenvalues in a domain perturbed as described above (with some additional geometric assumption on the section of the tube) using Almgren-type monotonicity formulas, Courant-Fischer min-max characterization for eigenvalues and blow-up analysis for scaled eigenfunctions. Specifically, they restrict to the case in which the perturbed eigenvalues are converging to a simple eigenvalue of the limit problem.
Due to the local nature of the singular perturbation taken into account, the eigenfunctions' local behavior at $0\in\partial \Omega$ (namely the point where the thin branch is attached) plays a crucial role. This behavior
can be described as follows (see \cite{Bers1955} or \cite[Theorem 1.3]{FFT2011}): if $\varphi_N$ is an eigenfunction of \eqref{eq:str_dir_eigen}, there exists $k\in\N\setminus \{0\}$ such that    
\begin{equation}\label{eq:conv_psi_gamma}
	\frac{\varphi_N(r x)}{r^k}\to \psi_k(x), 
	\quad \text{in }C^{1,\alpha}(\overline{B_1^+}) ~ \text{as }r\to0,~\text{for some }\psi_k\in \mathbb{P}^k_{\textup{odd}},
\end{equation}
where $\mathbb{P}^k_{\textup{odd}}$ denotes the space of harmonic homogeneous polynomials of degree $k$, odd with respect to the last variable $x_d$. We point out that the polynomials in the class $\mathbb{P}^k_{\textup{odd}}$, restricted to the $(d-1)$-dimensional unit sphere, are spherical harmonics (i.e. eigenfunctions of the spherical Laplacian) vanishing on $\{x_d=0\}$.
If we denote 
\begin{equation}\label{eq:dom_Pi}
	\Pi:=\R^d_+\cup\Sigma\cup T,\quad\text{where }T:=\Sigma \times(-\infty,0)
\end{equation}
and by $\mathcal D^{1,2}(\Pi)$ the completion of $C^\infty_c(\Pi)$ with respect to the $L^2$ norm of the gradient, the main result \cite[Theorem 1.1]{FO} establishes that, if $\Sigma$ is starshaped with respect to the origin and if $\lambda_N$ is a simple eigenvalue of \eqref{eq:str_dir_eigen}, then 
\begin{equation}\label{eq:risultatoFO}
\lambda_N^\epsilon=\lambda_N-C_{k,\Sigma}\,\e^{d-2+2k}+o(\e^{d-2+2k}),\quad\text{as }\e\to 0,
\end{equation}
where 
\begin{equation}\label{eq:risultatoFOsegue}
C_{k,\Sigma} =-2\, \inf_{u\in \mathcal D^{1,2}(\Pi)} \left\{ \frac12 \int_\Pi |\nabla u|^2\dx -\int_{\Sigma} u\dfrac{\partial \psi_k}{\partial x_d}\dx' \right\}>0.
\end{equation}

Let us now briefly comment on this result. From \eqref{eq:risultatoFO}, one can see that the local nature of the perturbation mainly emerges in the exponent $k$ of the radius of the tube's section $\e$: in this sense, the vanishing order $k$ of the unique limit eigenfunction (up to multiplicative constants) at the junction point determines the rate of convergence of the perturbed eigenvalue. The second factor which influences the asymptotic expansion \eqref{eq:risultatoFO} is the positive constant $C_{k,\Sigma}$, whose variational characterization \eqref{eq:risultatoFOsegue} sheds some light on its nature. As already noticed for the first time in \cite{AFT2014}, this coefficient has to do with the ability of a membrane to respond to a vertical force acting on it.
The main drawback of \cite{FO} is the geometric assumptions on the tube's section $\Sigma$ and the hypothesis of simplicity of the limit eigenvalue $\lambda_N$.

In this paper, we mainly address these two open questions. Concerning nontrivial multiplicity of eigenvalues, less is available in literature as compared to the simple case (we refer to \cite{FO} for the state of the art in this last instance). In this regard, we mention the work by Taylor \cite{Taylor2013} which provides an estimate for the eigenvalue variation in a similar context. In this work, the distance between the perturbed and the limit eigenvalue is estimated by $C\e^a$, where the rate $a$ is independent of any eigenvalue and the constant $C$ depends only on the distance between the limit eigenvalue to the nearby ones. The same author provides a similar result in collaboration with Collins (see \cite{CollinsTaylor2018}) for problems with mixed Dirichlet--Neumann boundary conditions if the tube is attached at a point where Dirichlet condition is imposed.  
Apart from these, no sharper result on eigenbranches is available in literature, for which eigenbranches could be distinguished each other by their different asymptotic behavior as $\e\to0$. Splitting of eigenbranches is proved in \cite[Chapter 3]{Gadylshin2005tubi} but only in dimension $2$. To the best of our knowledge, no other result is available on this issue. Nevertheless, we strongly believe that this is a relevant topic of investigation: multiple eigenvalues appear in domains with symmetries and this is often the case in applications.

Before stating our main results, we give the fundamental definition of the paper together with some remarks on it.  
For $f\in C^1\left(\overline{B_{r_0\e}^+}\right)$
we 
introduce the functional
\begin{equation}\label{eq:functional_torsion}
	 J_{\e\Sigma,f}^{\Omega_\e}(u)=J_f^\e(u):=\frac{1}{2}\int_{\Omega_\e}\abs{\nabla u}^2\dx - \int_{\e\Sigma}u\dfrac{\partial f}{\partial x_d}\dx',
\end{equation}
defined for $u\in H^1_0(\Omega_\e)$\footnote{we note that $J_f^\e$ can be defined for less regular functions $f$: it may be sufficient that $\frac{\partial f}{\partial x_d} \in \left(H^{1/2}_{00}\right)'$ provided the integral $\int_{\e\Sigma}u\frac{\partial f}{\partial x_d}\dx'$ is meant as a duality product. }. 

\begin{definition}\label{def:torsion_functional}
For any $f\in C^1\left(\overline{B_{r_0\e}^+}\right)$
	we call the \emph{thin $f$-torsional rigidity of $\e\Sigma$ relative to $\Omega_\e$} the following quantity
	\begin{align*}
\mathcal{T}_{\Omega_\e}(\e\Sigma,f):&=-2\inf_{u\in H^1_0(\Omega_\e)} J_f^\e(u), \\
 &=\sup_{u\in H^1_0(\Omega_\e)}\left\{2\int_{\e\Sigma}u\frac{\partial f}{\partial x_d}\dx'-\int_{\Omega_\e}\abs{\nabla u}^2\dx  \right\}.
	\end{align*}
	If $\frac{\partial f}{\partial x_d}=1$ on $\e\Sigma$, we denote $\mathcal{T}_{\Omega_\e}(\e\Sigma):=\mathcal{T}_{\Omega_\e}(\e\Sigma,f)$ and we call it  the \emph{thin torsional rigidity of $\e\Sigma$ relative to $\Omega_\e$}. 
\end{definition}
If we consider $\e\Sigma$ as a variable and $\Omega_\e$ and $f$ as parameters, 
$\mathcal T_{\Omega_\e}(\e\Sigma,f)$ is a set-function and will play a crucial role in our analysis. It will be the right quantity to face the perturbation theory of eigenvalues in this framework as soon as $f$ is a suitable relative eigenfunction. 
Realizing this is
one of the main novelty of this work. 
Broadly speaking, this notion of torsional rigidity of the perturbing set plays a similar role as that of the capacity of a set in \cite{ALM2022} (see \Cref{subsec:mixed} for a more detailed explanation). 
In \Cref{sec:facts} we report some basic results concerning this quantity, such as existence and uniqueness of a minimizer for $J_f^\e$, equivalent formulations, monotonicity properties and asymptotic behavior as $\e\to 0$. Here we just mention that there exists a unique function $U^{\Omega_\e}_{\e\Sigma,f}=U_f^\e\in H^1_0(\Omega_\e)$ (depending also on $\Sigma$) achieving $\mathcal{T}_{\Omega_\e}(\e\Sigma,f)$. Moreover, it weakly satisfies
	\begin{equation*}
	\begin{bvp}
		-\Delta U_f^\e &=0, &&\text{in }\Omega_\e\setminus \e\Sigma \\
		U_f^\e &=0, &&\text{on }\partial\Omega_\e, \\
	\frac{\partial U_f^\e\restr{\Omega}}{\partial x_d}-\frac{\partial U_f^\e\restr{T_\e}}{\partial x_d}&=-\frac{\partial f}{\partial x_d}, &&\text{on }\e\Sigma,
	\end{bvp}
\end{equation*}
in the sense that
\begin{equation}\label{eq:tors_function_weak}
	\int_{\Omega_\e}\nabla U_f^\e\cdot\nabla \varphi\dx=\int_{\e\Sigma}\varphi \frac{\partial f}{\partial x_d}\dx',\quad\text{for all }\varphi\in H^1_0(\Omega_\e).
\end{equation}
In addition, one can observe that the map
	\[
			f\mapsto U^\e_f
	\]
	is linear. When $\frac{\partial f}{\partial x_d}=1$ on $\e\Sigma$, we drop the index $f$ and we denote by $U^\e$ the unique function in $H^1_0(\Omega_\e)$ achieving $\mathcal{T}_{\Omega_\e}(\e\Sigma)$. We call $U_f^\e$ and $U^\e$ the \emph{thin $f$-torsion function} and \emph{thin torsion function} of $\e\Sigma$ (relative to $\Omega_\e$), respectively.

\begin{remark}
Classically, the definition of the \emph{torsional rigidity} is the $L^1$ norm of the solution $u_\Omega$ to the problem 
\[
 \begin{cases}
  -\Delta u=1 &\text{in } \Omega\\
  u=0 &\text{on } \partial\Omega,
 \end{cases}
\]
that is $\mathcal T(\Omega)=\int_\Omega u_{\Omega}\,dx$. We stress that $u_\Omega$ solves 
\[
\int_\Omega \nabla u_\Omega \cdot \nabla v \dx= \int_\Omega v \dx \quad \text{for all }v\in H^1_0(\Omega).
\]
Its name is due to the theory of elasticity, where the quantity plays a role in the model of a twisted beam subjected to a torque. At the same time, it plays a role in the model of the displacement of a membrane subjected to a pressure, as well. In this latter case, if the pressure is concentrated on a $(N-1)$-dimensional regular oriented manifold $\Sigma\subset\Omega$ and the membrane is hanged to a fixed support (homogeneous Dirichlet boundary conditions), one expects to end up (at least formally) with the model 
\[
\int_\Omega \nabla u_{\Omega,\Sigma} \cdot \nabla v \dx= \int_\Sigma v \dx \quad \text{for all }v\in H^1_0(\Omega).
\]
This gives rise to the aforementioned \emph{thin torsional rigidity} (see also the discussion in the subsequent Subsection \ref{subsec:mixed}), where the $L^1(\Omega)$ norm is replaced by the $L^1(\Gamma)$ norm. Among many others, we refer to the comprehensive book \cite{henrot2018} for a detailed exposition on the classical notion of torsional rigidity.
\end{remark}

Coming back to our problem, let $\lambda_N$ be an eigenvalue of \eqref{eq:str_dir_eigen} with multiplicity $m\geq 1$ and let us denote by $E(\lambda_N)\sub H^1_0(\Omega)$ the associated $m$-dimensional eigenspace. As already mentioned in the introduction, for $i\in\{1,\dots,m\}$,
\begin{equation*}
	\lambda_{N+i-1}^\e\to \lambda_N, \quad \mbox{as } \e\to0.
\end{equation*} 
Therefore we have exactly $m$ eigenvalue branches departing from the multiple limit eigenvalue $\lambda_N$ and some of them may a priori coincide. We investigate the asymptotic behaviors of $\lambda_{N+i-1}^\e$ as $\e\to0$.
In order to find good approximations for perturbed eigenvalues
we use a slight modification of a lemma by G. Courtois \cite{Courtois1995}, itself based on the work by Y. Colin de Verdi\'{e}re \cite{ColindeV1986}.  
It establishes essentially the possibility to approximate small eigenvalues of a quadratic form with eigenvalues of the same form restricted to a finite dimensional subspace of the form domain. One pays an error which can be estimated in terms of spectral projections.
For completeness, we report its statement in \Cref{p:appEV}. 
Good approximations for perturbed eigenvalues rely on good approximations for perturbed eigenfunctions. These will be suitable modifications of the limit eigenfunctions corresponding to the limit eigenvalue $\lambda_N=\cdots=\lambda_{N+m-1}$. If $\varphi$ is an eigenfunction relative to $\lambda_N$, its best approximation will be  
\[
	\Pi_\e\varphi=\Phi^\e_{\varphi}:=\varphi+U^\e_\varphi, \quad\text{with }\varphi\in E(\lambda_N),
\]
where $U^\e_\varphi$ is the thin $\varphi$-torsion function.
Indeed, in view of \eqref{eq:tors_function_weak}, the function $\Phi^\e_\varphi\in H^1_0(\Omega_\e)$ weakly solves
\[
	\begin{bvp}
		-\Delta \Phi^\e_\varphi&=\lambda_N \varphi,&&\text{in }\Omega_\e, \\
		\Phi^\e_\varphi&=0,&&\text{on }\partial\Omega_\e.
	\end{bvp}
\]
This essentially means that $\Phi_\varphi^\e$ is in the domain of the perturbed operator, but $-\Delta\Phi_\varphi^\e$ acts as $-\Delta \varphi$ in the sense of distributions on $H^1_0(\Omega)$. 

If we apply Proposition \ref{p:appEV} in a suitable way (see Section \ref{sec:appLemma} for the details) we obtain our first main result. We state it in the following  

\begin{theorem}\label{thm:approxEVs} For $i\in\{1,\dots,m\}$,
\begin{equation}\label{eq:asymptEV}
	\lambda_{N+i-1}^\e=\lambda_N-\mu_i^\e+o(\chi_\e^2) \mbox{ as }\e\to0,
\end{equation}
where
\begin{equation*}
	\chi_\e^2:=\sup\{\mathcal T_{\Omega_\e}(\e\Sigma,u)\,:\,u\in E(\lambda_N)\mbox{ and } \|u\|_{L^2(\Omega)}=1\}
\end{equation*}
and $\{\mu_i^\e\}_{i=1}^m$ are the eigenvalues (taken in non-increasing order) of the quadratic form $r_\e$, defined for $u,v\in E(\lambda_N)$ as
\begin{equation*}
	r_\e(u,v) := \int_{\Omega_\e}\nabla U_u^\e\cdot \nabla U_v^\e\dx+\lambda_N\int_{\Omega_\e} U_u^\e\,U_v^\e\dx,
\end{equation*}
where $U_u^\e$ ($U_v^\e$) is the thin $u$-torsion function of $\e\Sigma$ (the thin $v$-torsion function of $\e\Sigma$, respectively).
\end{theorem}

\begin{remark}\label{r:simple}
 In fact, one can obviously even consider $m=1$, for this is allowed in Proposition \ref{p:appEV}. In this case, Theorem \ref{thm:approxEVs} readily implies the main result in \cite{FO} (see \eqref{eq:risultatoFO} and \eqref{eq:risultatoFOsegue}), if supplied with the blow-up analysis for $\mathcal T_{\Omega_\e}(\e\Sigma,f)$ given in Theorem \ref{thm:blow_up1}. 
\end{remark}

Although Theorem \ref{thm:approxEVs} provides a good approximation for perturbed eigenvalues when the limit one is simple, it is not exhaustive for multiple ones. For instance, for some $i$ it may be $\mu_i^\e=o(\chi_\e^2)$: in this case \eqref{eq:asymptEV} reduces to $\lambda_{N+i-1}^\e-\lambda_N=o(\chi_\e^2)$, providing no sharp asymptotic behavior but just an estimate. Nevertheless, we expect that eigenbranches' asymptotic rates will depend on local behaviors at $0$ of properly chosen limit eigenfunctions. In order to improve Theorem \ref{thm:approxEVs} in this way, we introduce the following proposition. It sheds light on the proper choice of the limit eigenbasis: the limit eigenspace can be uniquely split in several subspaces where eigenfunctions share the same vanishing order at $0$. We call it the \emph{order decomposition} of $E(\lambda_N)$.  

\begin{proposition}\label{prop:DecompESintro} There exists an integer $p\geq 1$, a decomposition of $E(\lambda_N)$ into a sum of orthogonal subspaces
\[E(\lambda_N)=E_1\oplus\dots\oplus E_p\] 
and an associated finite increasing sequence of integers
\[0<k_1<\dots<k_p\]
such that, for all $1\le j \le p$, a function $\varphi \in E_j\setminus\{0\}$ has vanishing order $k_j$ at $0$, that is
\[
	\frac{\varphi(rx)}{r^{k_j}}\to \psi_{k_j}(x)\quad\text{in }C^{1,\alpha}(\overline{B_1^+}),~\text{as }\e\to 0,
	\]
	for some harmonic polynomial $\psi_{k_j}$, homogeneous of degree $k_j$ and odd with respect to $x_d$.
In addition, such a decomposition is unique.  
\end{proposition}

Secondly, we will need
a blow-up analysis for torsion functions. Let us describe in a few words our procedure. By the change of variables $x\mapsto \e x$, we zoom in closely to the origin: the perturbed domain $\Omega_\e$ is transformed into 
\[
	\frac{1}{\e}\Omega_\e=\frac{1}{\e}\Omega\cup \Sigma \cup \frac{1}{\e} T_\e,
\]
which is intuitively \enquote{converging}, to $\Pi$ (defined in \eqref{eq:dom_Pi}), as $\e\to 0$. Given the order decomposition as in \Cref{prop:DecompESintro} and fixed $j\in\{1,\dots,p\}$, we consider the map
\begin{equation*}
		\mathfrak{B}_j \colon  E_j\to \mathbb{P}^{k_j}_{\textup{odd}},
\end{equation*}
where $\mathfrak{B}_j\varphi$ is the harmonic homogeneous polynomial of degree $k_j$ which describes the local behavior of $\varphi\in E_j$ near the origin, i.e.
\[
	\frac{\varphi(rx)}{r^{k_j}}\to \left(\mathfrak{B}_j\varphi\right)(x)\quad\text{in }C^{1,\alpha}(B_1^+),~\text{as }r\to 0.
\]
We observe that for $\varphi\in E_j$ the aforementioned change of variables gives
\[
\begin{aligned}
		\mathcal{T}_{\Omega_\e}(\e\Sigma,\varphi)&=-2\left[\frac{1}{2}\int_{\Omega_\e}\abs{\nabla U_\varphi^\e}^2\dx-\int_{\e\Sigma}U_\varphi^\e\frac{\partial \varphi}{\partial x_d}\dx' \right] \\
		&=-2\e^{d-2+2k_j}\left[\frac{1}{2}\int_{\frac{1}{\e}\Omega_\e}|\nabla\hat{U}_\varphi^\e|^2\dx-\int_{\Sigma}\hat{U}_\varphi^\e\frac{\partial \hat{\varphi}^\e}{\partial x_d}\dx' \right],
\end{aligned}
\]
where
\[
	\hat{U}_\varphi^\e(x):=\frac{U_\varphi^\e(\e x)}{\e^{k_j}}\quad\text{and}\quad\hat{\varphi}^\e(x):= \frac{\varphi(\e x)}{\e^{k_j}}.
\]
In view of this, it is reasonable to investigate the behavior of $\hat{U}_\varphi^\e$ as $\e\to 0$ and to expect that $\e^{-d+2-2k_j}\mathcal{T}_{\Omega_\e}(\e\Sigma,\varphi)$ admits a nontrivial, finite, limit. This motivates the following quantity:
\begin{equation}\label{eq:def_blowup_torsion}
	\mathcal{T}_\Pi(\Sigma,\Psi):=-2\inf\left\{ \frac{1}{2}\int_\Pi\abs{\nabla u}^2\dx-\int_\Sigma u\frac{\partial \Psi}{\partial x_d}\dx'\colon u\in \mathcal{D}^{1,2}(\Pi) \right\},
\end{equation}
which we define for $\Psi\in C^1(\overline{B_{r_0}^+})$\footnote{We point out that a little abuse of notation has been made, since this minimization is made within $\mathcal{D}^{1,2}(\Pi)$, which is strictly larger than $H^1_0(\Pi)$.}. If $\frac{\partial \Psi}{\partial x_d}=1$ on $\Sigma$, we denote $\mathcal{T}_\Pi(\Sigma):=\mathcal{T}_\Pi(\Sigma,\Psi)$. Therefore, our \Cref{thm:blow_up1} will establish
\[
	\e^{-d+2-k_j}\mathcal{T}_{\Omega_\e}(\e\Sigma,\varphi)\to \mathcal{T}_\Pi (\Sigma,\mathfrak{B}_j\varphi)\quad\text{as }\e\to 0,~\text{for all }\varphi\in E_j,
\]
for all $j=1,\dots,p$.  Hereafter we denote by $U^\Pi_{\Sigma,\varphi}$ the function achieving $\mathcal{T}_\Pi(\Sigma,\mathfrak{B}_j\varphi)$, for $\varphi\in E_j$.

Taking into account Proposition \ref{prop:DecompESintro} and this blow-up analysis, we are able to prove the following result about the sharp asymptotic behavior of perturbed eigenvalues. Before stating it, we need the following notation. For $1\le j\le p$, we let $E_j$ as in \Cref{prop:DecompESintro} and we write 
	\[m_{j}:=\mathrm{dim}\,(E_{j}),\]
	so that 
	\[m=m_1+\dots+m_j+\dots+m_p,\]
	whereas we denote by
	\[\mu_{j,1}\ge\dots\ge\mu_{j,\ell}\ge\dots\ge\mu_{j,m_j}>0\]
	the eigenvalues of the bilinear form
	\begin{equation*}
		\mathfrak T_j(u,v)=\int_{\Pi} \nabla U_u\cdot \nabla U_v\dx \quad \text{defined for }u,v\in E_j\subseteq E(\lambda_N),
	\end{equation*}
	where $U_u=U^\Pi_{\Sigma,u}$ and $U_v=U^\Pi_{\Sigma,v}$ achieve $\mathcal{T}_\Pi(\Sigma,\mathfrak{B}_j u)$ and $\mathcal{T}_\Pi(\Sigma,\mathfrak{B}_j v)$, respectively. 
We are now ready to state the main result of our paper.

\begin{theorem} \label{thm:orderEVs}
	For any $i\in\{1,\dots,m\}$, there holds
	\begin{equation}\label{eq:asymptEVOrder}
		\lambda_{N+i-1}^\e=\lambda_N-\mu_{j,\ell}\,\e^{d-2+2k_{j}}+o(\e^{d-2+2k_{j}}),\quad\mbox{as }\e\to0,
	\end{equation}
	where 
	\[
	 (j,\ell)=\begin{cases}
	       (1,i) \quad &\text{if } 1\leq i \leq m_1\\
	       (2,i-m_1) \quad &\text{if } m_1+1\leq i \leq m_1+m_2\\
	       \vdots & \vdots \\
	       (p,i-(m-m_p)) \quad &\text{if } m-m_p+1\leq i \leq m
	      \end{cases}
	\]
\end{theorem}

This result concludes our analysis on this class of problems: it provides sharp asymptotics for any perturbed eigenvalue.

Let us now analyze more in depth the particular case when $p=1$ and $k_1=1$. This is relevant since the gradient of limit eigenfunctions vanish at most on a subset of $\partial\Omega\cap \{x_d=0\}$ which has zero $(d-1)$-dimensional measure (see e.g. \cite{HHH99}). Thus, the limit eigenfunctions vanish with order $1$ for $\mathcal{L}^{d-1}$-almost every point in $\partial\Omega\cap \{x_d=0\}$. 
Broadly speaking, vanishing order 1 occurs generically with respect to the points in $\partial\Omega\cap \{x_d=0\}$.
We also observe that if $d=2$ $p=1$ can only uccur if $m=1$ (see \Cref{rem:MaxDim}). If $u\in E(\lambda_N)$, we have that
\[
	\mathfrak{B}_1 u=\frac{\partial u}{\partial x_d}(0)\,x_d,
\]
that is
\[
	\frac{u(rx)}{r}\to \frac{\partial u}{\partial x_d}(0)\,x_d,\quad\text{in }C^{1,\alpha}(\overline{B_1^+}),~\text{as }\e\to 0.
\]
Hence, by linearity there holds
\[
	U^\Pi_{\Sigma,u}=\frac{\partial u}{\partial x_d}(0)\,U^\Pi_{\Sigma,x_d}.
\]
Therefore, we can write down the quadratic form 
\[
	\mathfrak{T}_1(u,v)=\frac{\partial u}{\partial x_d}(0)\frac{\partial v}{\partial x_d}(0)\int_\Pi\abs{\nabla U^\Pi_{\Sigma,x_d}}^2\dx=\frac{\partial u}{\partial x_d}(0)\frac{\partial v}{\partial x_d}(0)\,\mathcal{T}_\Pi(\Sigma).
\]
In particular, it is possible to choose an eigenbasis $\{\varphi_N,\dots,\varphi_{N+m-1}\}$ for $E(\lambda_N)$ which diagonalizes $\mathfrak{T}_1$, in such a way that
\[
	\mu_{1,i}=\left(\frac{\partial \varphi_{N+i-1}}{\partial x_d}(0)\right)^2\,\mathcal{T}_\Pi(\Sigma),\quad\text{for }i=1,\dots,m,
\]
thus leading to the following.
\begin{corollary}
	Let us assume \Cref{prop:DecompESintro} holds with $p=1$ and $k_1=1$. Then there exists a basis $\{\varphi_N,\dots,\varphi_{N+m-1}\}$ of $E(\lambda_N)$, orthonormal in $L^2(\Omega)$ and such that, for all $i\in\{1,\dots,m\}$, there holds
	\[
		\lambda_{N+i-1}^\e=\lambda_N-\left(\frac{\partial \varphi_{N+i-1}}{\partial x_d}(0)\right)^2\,\mathcal{T}_\Pi(\Sigma)\,\e^{d}+o(\e^d),\quad\text{as }\e\to 0.
	\]
\end{corollary}
We observe that this result recovers what the authors obtained in \cite{AFT2014} in the case $m=1$. Finally, we would like to emphasize another particular instance, which concerns the behavior of the first eigenvalue, being one of the most widely studied set functions. In this case, it is known that the limit eigenvalue is simple (i.e. $m=1$) and that the corresponding eigenfunctions have nonzero gradient on any regular boundary point, being them positive in $\Omega$ (i.e. $p=1$ and $k_1=1$). Hence, we have the following.
\begin{corollary}
	Let $\varphi_1\in H^1_0(\Omega)$ be a normalized eigenfunction corresponding to $\lambda_1$. Then there holds
	\[
		\lambda_1^\e=\lambda_1-\left(\frac{\partial \varphi_1}{\partial x_d}(0)\right)^2\,\mathcal{T}_\Pi(\Sigma)\,\e^d+o(\e^d),\quad\text{as }\e\to 0.
	\]
\end{corollary}
It is worth noticing that, in these two last results, the coefficients of the first term in the asymptotic expansion of the eigenvalue variation split as a product of two factors: one of them only depending on the behavior of the limit eigenfunctions at the origin and the other one only depending on the geometry of the set $\Sigma$. Moreover, we recall that in \cite{AFT2014} authors prove that the spherical shape for $\Sigma$ maximizes $\mathcal T_{\Pi}(\Sigma)$ among all sections with fixed measure.

The paper is organized as follows. \Cref{sec:facts} is devoted to present basic properties of the \emph{thin $f$-torsional rigidity} and contains preliminary result in view of the main theorems. 
\Cref{sec:perturbation} contains the proof of \Cref{thm:approxEVs}, whereas \Cref{sec:ramification} contains the proof of \Cref{thm:orderEVs}.

\subsection{Mixed Dirichlet--Neumann boundary conditions}\label{subsec:mixed}

The preceeding arguments apply also when we perturb problem \eqref{eq:str_dir_eigen} by prescribing that eigenfunctions satisfy homogeneous Neumann boundary conditions on $\e\Sigma$, in place of attaching a thin tube with section $\e\Sigma$ to the fixed domain $\Omega$.

The problem has been already studied in dimension $2$ in \cite{Gadyl'shin1992}, achieving a full asymptotic expansion of perturbed eigenvalues (see also \cite{AFL2018} for related results) and in any dimension by \cite{FNO2} but only for simple eigenvalues. This last paper is our starting point. 

Let us consider the weak form of the eigenvalue problem 
\begin{equation}\label{eq:mixedBCprobl}
 \begin{bvp}
  -\Delta \xi&= \tilde{\lambda}^\epsilon \xi, &&\text{in }\Omega\\
  \xi&=0, &&\text{on }\partial\Omega\setminus \e\Sigma\\
  \dfrac{\partial \xi}{\partial\nnu}&=0 &&\text{on } \e\Sigma.
 \end{bvp}
\end{equation}
To this aim, we set the functional framework as it appears in \cite{FNO2}, by introducing the space
$
 H^1_{0, \partial\Omega\setminus \e\Sigma}(\Omega),
$
 defined as the closure in $H^1(\Omega)$ of $C^\infty_c(\Omega\cup \e\Sigma)$. 
We say that $\tilde{\lambda}^\e\in\R$ is an \emph{eigenvalue} of \eqref{eq:mixedBCprobl} if there exists $\xi^\e \in  H^1_{0, \partial\Omega\setminus \e\Sigma}(\Omega)\setminus\{0\}$ (named \emph{eigenfunction}) such that 
\[
 \int_{\Omega}\nabla \xi^\e \cdot \nabla v \dx= \tilde{\lambda}^\e \int_{\Omega}  \xi^\e v\dx \quad \text{for all }v \in  H^1_{0, \partial\Omega\setminus \e\Sigma}(\Omega).
\]
For any $\e\in(0,1)$ there exists a non-decreasing sequence of positive eigenvalues
\[
	0<\tilde{\lambda}_1^\e<\tilde{\lambda}_2^\e\leq \cdots\leq \tilde{\lambda}_n^\e\leq \cdots\to +\infty.
\]
When $\e\to0$ the Neumann region disappears and the Dirichlet region covers the entire boundary. One then expects the eigenelements of the mixed Dirichlet-Neumann problem \eqref{eq:mixedBCprobl} to converge to the ones of the limit Dirichlet problem \eqref{eq:str_dir_eigen}. 
This problem revealed to be more involved than its counterpart, in which Neumann boundary condition are prescribed on a large part of $\partial\Omega$ and Dirichlet boundary conditions on a vanishing portion of it. A capacitary approach (such as the one developed in \cite{FNO1} for the case of disappearing Dirichlet region) turns out to be particularly effective when functions are required to vanish in \enquote{small} sets; this is basically related to the known fact that Sobolev spaces are not affected if their functions are prescribed to vanish on zero capacity sets. So, the case introduced in this subsection falls outside a capacitary context. Furthermore, to the best of our knowledge, in literature there is no analogue to the capacity, which can play a similar role in the converse case \eqref{eq:mixedBCprobl}. Because of this, in \cite{FNO2} the authors undertake the problem through another approach, based on Almgren-type monotonicity formula. In this work they prove the convergence of the perturbed spectrum, as $\e\to 0$, to the spectrum of the Dirichlet-Laplaciann (see \cite[Proposition 2.3]{FNO2}), that is 
\begin{equation}\label{eq:stability_mixed}
	\tilde{\lambda}_n^\epsilon\to \lambda_n \quad \text{as }\e\to0 \text{ for any }n\in\N\setminus\{0\}.
\end{equation}
Moreover, in case of simple limit eigenvalues, they provide an explicit asymptotic expansion of the perturbed eigenvalues, which is sharp only when the set $\Sigma$ fulfills certain geometric assumptions. In particular, they proved that, if $\Sigma$ is strictly starshaped with respect to the origin, if $\lambda_N$ is simple and if \eqref{eq:conv_psi_gamma} holds, then
\begin{equation}\label{eq:risultatoFNO2}
	\tilde{\lambda}_N^\e=\lambda_N-\tilde{C}_{k,\Sigma}\,\e^{d-2+2k}+o(\e^{d-2+2k}),\quad\text{as }\e\to 0,
\end{equation}
where
\begin{equation*}
	\tilde{C}_{k,\Sigma}:=-2\inf_{u\in \mathcal{D}^{1,2}(\R^d_+\cup\Sigma)}\left\{ \frac{1}{2}\int_{\R^d_+}\abs{\nabla u}^2\dx+\int_\Sigma u \frac{\partial \psi_k}{\partial\nnu}\dx' \right\}
\end{equation*}
and $\mathcal D^{1,2}(\R^d_+\cup\Sigma)$ denotes the completion of $C^\infty_c(\R^d_+\cup\Sigma)$ with respect to the $L^2$ norm of the gradient. We also observe that here $\nnu=(0,\dots,0,-1)$ and $\frac{\partial\psi_k}{\partial\nnu}=-\frac{\partial \psi_k}{\partial x_d}$.

With the very same method used to prove the results in the previous subsection, we are able to remove both the simplicity assumption on the limit eigenvalue $\lambda_N$ and the geometric assumption on $\Sigma$ and to prove a sharp asymptotic expansion for $\tilde{\lambda}_N^\e$ in the general case. In the same spirit of the previous case, here we are able to detect the proper quantity which measures the magnitude of the perturbation and, consequently, the stability of eigenvalues. It is still a notion of torsional rigidity and plays the role of perfect counterpart of the capacity in the present framework. Let us first introduce the functional
\begin{equation*}
	\tilde{J}^\Omega_{\e\Sigma,f}(u)=\tilde{J}^\e_f(u):=\frac{1}{2}\int_\Omega\abs{\nabla u}^2\dx+\int_{\e\Sigma}u\frac{\partial f}{\partial\nnu}\dx',
\end{equation*}
defined for $u\in H^1_{0,\partial\Omega\setminus\e\Sigma}(\Omega)$, where $f\in C^1(\overline{B_{r_0}^+})$. In the same lines as in the previous subsection, we introduce the notion of relative torsional rigidity of a set which is suitable for our problem.
\begin{definition}\label{def:torsion_functional_mixedbc}
	For any $f\in C^1(\overline{B_{r_0}^+})$ we call the \emph{boundary $f$-torsional rigidity of $\e\Sigma$ relative to $\Omega$} the following quantity
	\begin{align*}
		\mathcal{T}_{\Omega}(\e\Sigma,f):&=-2\inf_{u\in  H^1_{0, \partial\Omega\setminus \e\Sigma}(\Omega)}\tilde{J}^\e_f(u) \\
		&=-2\inf_{u\in  H^1_{0, \partial\Omega\setminus \e\Sigma}(\Omega)} \left\{ \frac12\int_{\Omega}\abs{\nabla u}^2\dx +\int_{\e\Sigma}u\dfrac{\partial f}{\partial \nnu}\dx'\right\}
	\end{align*}
	If $\frac{\partial f}{\partial \nnu}=-1$ on $\e\Sigma$, we denote $\mathcal{T}_{\Omega}(\e\Sigma):=\mathcal{T}_{\Omega}(\e\Sigma,f)$ and we call it  the \emph{boundary torsional rigidity of $\e\Sigma$ relative to $\Omega$}. 
\end{definition}
We point out that Definition \ref{def:torsion_functional} and Definition \ref{def:torsion_functional_mixedbc} are completely matching each other.  
As in the previous case, by standard minimization methods there exists $\tilde{U}^\e_f=\tilde U^{\Omega}_{\e\Sigma,f} \in H^1_{0, \partial\Omega\setminus \e\Sigma}(\Omega) \setminus \{0\}$ achieving $\mathcal{T}_\Omega(\e\Sigma,f)$. We call $\tilde U^\e_f$ the \emph{boundary $f$-torsion function of $\e\Sigma$, relative to $\Omega$}. In particular, $\tilde U^{\e}_f$ satisfies
\[
	\begin{bvp}
		-\Delta \tilde{U}^\e_f &=0, &&\text{in }\Omega, \\
		\tilde{U}^\e_f &=0, &&\text{on }\partial\Omega\setminus\e\Sigma, \\
		\frac{\partial \tilde{U}^\e_f}{\partial\nnu} &=-\frac{\partial f}{\partial \nnu}, &&\text{on }\e\Sigma
	\end{bvp}
\]
in a weak sense, that is $\tilde{U}^\e_f\in H^1_{0,\partial\Omega\setminus\e\Sigma}(\Omega)$ and
\begin{equation}\label{eq:weak_eq_mixed}
 \int_{\Omega}\nabla \tilde U^{\e}_f\cdot \nabla \varphi \dx=-  \int_{\e\Sigma} \varphi\dfrac{\partial f}{\partial \nnu}\dx' \quad \text{for all }  \varphi \in H^1_{0, \partial\Omega\setminus \e\Sigma}(\Omega).
\end{equation}

Let $\lambda_N$ be an eigenvalue to the problem \eqref{eq:mixedBCprobl} with multiplicity $ m$ and $ E( \lambda_N)$ be the associated eigenspace.
Our first result on this problem is the analogue of Theorem \ref{thm:approxEVs}.

\begin{theorem}\label{thm:approxEVsmixed} For $i\in\{1,\dots,m\}$,
\begin{equation*}
	\tilde\lambda_{N+i-1}^\e=\lambda_N-\tilde \mu_i^\e+o(\tilde \chi_\e^2) \mbox{ as }\e\to0,
\end{equation*}
where
\begin{equation*}
	\tilde \chi_\e^2:=\sup\{\mathcal T_{\Omega}(\e\Sigma,u)\,:\,u\in  E(\lambda_N)\mbox{ and } \|u\|_{L^2(\Omega)}=1\}
\end{equation*}
according to Definition \ref{def:torsion_functional_mixedbc}
and $\{\tilde\mu_i^\e\}_{i=1}^m$ are the eigenvalues (taken in non-increasing order) of the quadratic form $\tilde{r}_\e$, defined for $u,v\in E(\lambda_N)$ as
\begin{equation*}
	\tilde r_\e(u,v) := \int_{\Omega}\nabla \tilde U_u^\e\cdot \nabla \tilde U_v^\e+\lambda_N\int_{\Omega} \tilde U_u^\e\,\tilde U_v^\e,
\end{equation*}
where $\tilde U_u^\e$ ($\tilde U_v^\e$) is the boundary $u$-torsion function of $\e\Sigma$ (the boundary $v$-torsion function of $\e\Sigma$, respectively).
\end{theorem}
Given the order decomposition of $E(\lambda_N)$, as in \Cref{prop:DecompESintro}, and taking $\varphi\in E_j$ (for $j\in\{1,\dots,p\}$), it is possible to perform a blow-up analysis for a suitable rescaling of the $\varphi$-boundary torsion function $\tilde{U}^\e_\varphi$. Namely, following the steps as in \Cref{subsec:mosco} (outlined in the previous section), one can prove that
\[
	\e^{-d+2-2k_j}\tilde{U}^\e_\varphi(\e x)\quad\text{as well as}\quad \e^{-d+2-2k_k}\mathcal{T}_\Omega(\e\Sigma,\varphi)
\]
admit nontrivial, finite limits as $\e\to 0$. It is then natural to introduce the following quantity, defined for $\Psi\in C^1(\overline{B_1^+})$,
\begin{equation}\label{eq:torsion_blow_up}
	\mathcal{T}_{\R^d_+}(\Sigma,\Psi):=-2\left\{ \frac{1}{2}\int_{\R^d_+}\abs{\nabla u}^2\dx+\int_\Sigma u\frac{\partial\Psi}{\partial\nnu}\dx'\colon u\in \mathcal{D}^{1,2}(\R^d_+\cup \Sigma) \right\}.
\end{equation}
If $\frac{\partial\Psi}{\partial\nnu}=-1$ on $\Sigma$ we denote $\mathcal{T}_{\R^d_+}(\Sigma):=\mathcal{T}_{\R^d_+}(\Sigma,\Psi)$.

Hence, we are able to prove the following (which is the analogous of \Cref{thm:blow_up1}).
\begin{lemma}\label{l:blowup_mixed}
 Let $j\in\{1,\dots,p\}$ and let $\varphi\in E_j$. Then
	\begin{equation*}
		\mathcal{T}_{\Omega}(\e\Sigma,\varphi)=\e^{d-2+2k}\mathcal{T}_{\R^d_+}(\Sigma,\mathfrak{B}_j\varphi)+o(\e^{d-2+2k_j}),\quad\text{as }\e\to 0.
	\end{equation*}
	In addition
	\[
		\e^{-d+2-k_j}\tilde{U}^\e_\varphi(\e x)\to \tilde{U}^{\R^d_+}_{\Sigma,\varphi},\quad\text{in }\mathcal{D}^{1,2}(\R^d_+\cup\Sigma),~\text{as }\e\to 0,
	\]
	where $\tilde{U}^{\R^d_+}_{\Sigma,\varphi}\in \mathcal{D}^{1,2}(\R^d_+\cup \Sigma)$ denotes the function achieving $\mathcal{T}_{\R^d_+}(\Sigma,\mathfrak{B}_j\varphi)$.
\end{lemma}

\begin{remark}\label{r:simplemixed}
 Of course, even in this case one can consider $m=1$, for this is allowed in Proposition \ref{p:appEV}. In this way, Theorem \ref{thm:approxEVsmixed} provides immediately the main result as stated in \cite{FNO2} if supplied with \Cref{l:blowup_mixed}.
\end{remark}

Taking into account the order decomposition Proposition \ref{prop:DecompESintro} for $E(\lambda_N)$, and the blow-up analysis for scaled torsion functions, we are able to improve the result of Theorem \ref{thm:approxEVsmixed}. Before stating the main theorem, we need the following notation. For $1\le j\le p$, we let $E_j$ as in \Cref{prop:DecompESintro} and we recall that $m_j=\mathrm{dim}\,(E_j)$ so that 
\[m=m_1+\dots+m_j+\dots+m_p.\]
Moreover, we denote by
	\[\tilde{\mu}_{j,1}\ge\dots\ge\tilde{\mu}_{j,\ell}\ge\dots\ge\tilde{\mu}_{\ell,m_j}>0\]
the eigenvalues of the bilinear form
\begin{equation*}
	\tilde{\mathfrak T}_j(u,v)=\int_{\R^d_+} \nabla \tilde{U}_u\cdot \nabla \tilde{U}_v\dx \quad \text{defined for }u,v\in E_j\subseteq E(\lambda_N),
\end{equation*}
where $\tilde{U}_u=\tilde{U}^{\R^d_+}_{\Sigma,u}$ and $\tilde{U}_v=\tilde{U}^{\R^d_+}_{\Sigma,v}$ achieve, respectively, $\mathcal{T}_{\R^d_+}(\Sigma,\mathfrak{B}_j u)$ and $\mathcal{T}_{\R^d_+}(\Sigma,\mathfrak{B}_j v)$. We are now ready to state the main result of this section, which is the analogue of \Cref{thm:orderEVs}.
	\begin{theorem}\label{thm:orderEVsmixed}
		For any $i\in\{1,\dots,m\}$, there holds
		\begin{equation}
			\tilde{\lambda}_{N+i-1}^\e=\lambda_N-\tilde{\mu}_{j,\ell}\,\e^{d-2+2k_{j}}+o(\e^{d-2+2k_{j}}),\quad\mbox{as }\e\to0,
		\end{equation}
		where 
		\[
		(j,\ell)=\begin{cases}
			(1,i) \quad &\text{if } 1\leq i \leq m_1\\
			(2,i-m_1) \quad &\text{if } m_1+1\leq i \leq m_1+m_2\\
			\vdots & \vdots \\
			(p,i-(m-m_p)) \quad &\text{if } m-m_p+1\leq i \leq m
		\end{cases}
		\]
	\end{theorem}

For simplicity of exposition, we do not present here the proofs of \Cref{thm:approxEVsmixed} and \ref{thm:orderEVsmixed}, since they follow step by step the proofs of the case of domains with handles attached. It will be sufficient to set all the arguments in the appropriate functional setting, as described above.

Also for this kind of perturbation, we think it is interesting to understand what happens in some particular cases, similarly to what we described for the attachment of a thin tube in the previous section. More precisely, reasoning in a completely analogous way, if $p=1$ and $k_1=1$ one can see that it is possible to find an eigenbasis $\{\varphi_N,\dots,\varphi_{N+m-1}\}\sub E(\lambda_N)$, orthonormal in $L^2(\Omega)$, in such a way that
\[
	\tilde{\mu}_{1,i}=\left(\frac{\partial \varphi_{N+i-1}}{\partial\nnu}(0)\right)^2\,\mathcal{T}_{\R^d_+}(\Sigma)\quad\text{for }i=1,\dots,m.
\]
Before stating the result, we would like to observe that, in view of the characterization of the half-laplacian $(-\Delta_{\R^{d-1}})^{\frac{1}{2}}$ on $\R^{d-1}=\partial\R^d_+$ as a Dirichlet-to-Neumann map on $\R^d_+$, on can easily see that the quantity $\mathcal{T}_{\R^d_+}(\Sigma)$ coincides with $\frac{1}{2}$-fractional torsional rigidity of $\Sigma$ in $\R^{d-1}$. Namely $\mathcal{T}_{\R^d_+}(\Sigma)=\mathcal{T}_{\R^{d-1}}^{\frac{1}{2}}(\Sigma)$, where
\[
	\mathcal{T}_{\R^{d-1}}^{\frac{1}{2}}(\Sigma):=-2\inf\left\{ \frac{1}{2}\norm{u}_{\mathcal{D}^{\frac{1}{2},2}(\R^{d-1})}^2-\int_\Sigma u\colon u\in \mathcal{D}^{\frac{1}{2},2}_0(\Sigma) \right\}
\]
where $\mathcal{D}^{\frac{1}{2},2}_0(\Sigma)$ denotes the completion of $C_c^\infty(\Sigma)$ with respect to the norm
\[
	\norm{u}_{\mathcal{D}^{\frac{1}{2},2}(\R^{d-1})}:=\left( \frac{1}{(2\pi)^{\frac{d-1}{2}}}\int_{\R^{d-1}}\abs{\zeta}\abs{\hat{u}(\zeta)}^2\d\zeta \right)^{\frac{1}{2}}.
\]
Here by $\hat{u}$ we denote the (normalized) Fourier transform of $u$ in $\R^{d-1}$. We thus have the following.
\begin{corollary}
	Let us assume \Cref{prop:DecompESintro} holds with $p=1$ and $k_1=1$. Then there exists a basis $\{\varphi_N,\dots,\varphi_{N+m-1}\}$ of $E(\lambda_N)$, orthonormal in $L^2(\Omega)$ and such that, for all $i\in\{1,\dots,m\}$, there holds
	\[
	\tilde{\lambda}_{N+i-1}^\e=\lambda_N-\left(\frac{\partial \varphi_{N+i-1}}{\partial \nnu}(0)\right)^2\,\mathcal{T}_{\R^{d-1}}^{\frac{1}{2}}(\Sigma)\,\e^{d}+o(\e^d),\quad\text{as }\e\to 0.
	\]
\end{corollary}
We can also investigate, as a remarkable instance, the perturbation of the first eigenvalue and obtain the following.
\begin{corollary}
	Let $\varphi_1\in H^1_0(\Omega)$ be a normalized eigenfunction corresponding to $\lambda_1$. Then there holds
	\[
	\lambda_1^\e=\lambda_1-\left(\frac{\partial \varphi_1}{\partial x_d}(0)\right)^2\,\mathcal{T}_{\R^{d-1}}^{\frac{1}{2}}(\Sigma)\,\e^d+o(\e^d),\quad\text{as }\e\to 0.
	\]
\end{corollary}

\bigskip

The case of mixed Dirichlet--Neumann boundary conditions possesses additional features worthy to come to light. Firstly, let us recall the following definition, already introduced in \Cref{subsec:mixed}.
\begin{definition}\label{def:torsion_energy}
 Let $\Gamma\sub\partial\Omega$ be a relatively open set. We call \emph{boundary torsional rigidity of $\Gamma$ relative to $\Omega$} the quantity
 \[
 	\mathcal{T}_\Omega(\Gamma):=-2\inf\left\{ \frac{1}{2}\int_\Omega\abs{\nabla u}^2\dx-\int_\Gamma u\ds\colon u \in H^1_{0,\partial\Omega\setminus\Gamma}(\Omega) \right\}
 \]
  which coincides to the energy of the unique weak solution of the problem
 \begin{equation}\label{eq:Ueps1}
  \begin{bvp}
  -\Delta U_{\Gamma}&=0, &&\text{in }\Omega\\
  \frac{\partial U_\Gamma}{\partial\nnu} &=1, &&\text{on }\Gamma\\
  U_{\Gamma}&=0, &&\text{on }\partial\Omega\setminus\Gamma.
 \end{bvp} 
 \end{equation}
\end{definition}
We also recall that $H^1_{0,\partial\Omega\setminus\Gamma}(\Omega)$ denotes the closure of $C_c^\infty(\Omega\cup\Gamma)$ in $H^1(\Omega)$. We would like to stress that Problem \eqref{eq:Ueps1} has to do with the so-called \emph{boundary torsional rigidity of $\Omega$} as it is introduced in \cite{BrascoGonzalezIspizua2022}, there denotated by $T(\Omega,\delta)$. 
As explained in \cite[Section 2]{BrascoGonzalezIspizua2022}, $T(\Omega;\delta)$ is modeled on the trace Sobolev embedding $W^{1,2}(\Omega)\hookrightarrow L^1(\partial\Omega)$ and it is closely related to the Steklov eigenvalue problem. In this case it is worthwhile to mention that, equivalently
 \begin{equation}\label{eq:torsion_mixed_sup}
 \mathcal T_\Omega(\Gamma) = \sup_{\varphi\in H^1_{0,\partial\Omega\setminus\Gamma}(\Omega)\setminus \{0\}} \dfrac{\displaystyle \left(\int_{\Gamma} \varphi \ds\right)^2}{\displaystyle \int_{\Omega}|\nabla \varphi|^2\dx}, 
\end{equation}
for it is related to the best constant for the Sobolev embedding $H^1_{0,\partial\Omega\setminus\Gamma}(\Omega)\hookrightarrow L^1(\Gamma)$. Moreover, $\mathcal T_{\Omega}(\Gamma)$ is related to the lowest of the so-called  \emph{Dirichlet--Steklov eigenvalues}. 
As pointed out in \cite{Seo2021} (see also \cite{Agranovich2006}), the so-called \emph{Dirichlet--Steklov eigenvalue problem}
\[
 \begin{bvp}
  -\Delta u&=0,& &\text{in }\Omega\\
  u&=0,& &\text{on }\partial\Omega\setminus \Gamma\\
  \dfrac{\partial u }{\partial \nnu}&=\sigma u,&&\text{on }\Gamma
 \end{bvp}
\]
is equivalent to the eigenvalue problem of the Dirichlet-to-Neumann operator, which in fact admits a sequence of positive eigenvalues
\[
 0< \sigma_1(\Omega,\Gamma) \leq \sigma_2(\Omega,\Gamma)\leq \ldots \to +\infty.
\]
The lowest of them has the following variational characterization
\begin{equation}\label{eq:sigma1}
 \sigma_1(\Omega,\Gamma) = \inf\left\{ \dfrac{\displaystyle \int_\Omega |\nabla u|^2\dx}{\displaystyle \int_\Gamma u^2\ds}: u\in H^1_{0,\partial\Omega\setminus\Gamma}(\Omega)\setminus\{0\}\right\}.
\end{equation}
By definition and \eqref{eq:torsion_mixed_sup}, we have
\begin{align*}
 \dfrac1{\mathcal T_\Omega(\Gamma)} &= \dfrac{\displaystyle \int_\Omega |\nabla U_\Gamma|^2\dx}{\displaystyle \left(\int_\Gamma U_\Gamma\ds\right)^2} \geq \dfrac{\displaystyle \int_\Omega |\nabla U_\Gamma|^2\dx}{\displaystyle \mathcal{L}^{d-1}(\Gamma)\int_\Gamma {U_\Gamma}^2\ds}\geq \dfrac{\sigma_1(\Omega,\Gamma)}{\mathcal{L}^{d-1}(\Gamma)} 
\end{align*}
where we applied Cauchy-Schwarz Inequality to gain the first inequality and \eqref{eq:sigma1} to reach the second one. Summing up, we obtain 
\begin{equation}\label{eq:Polyatype}
 \mathcal T_\Omega(\Gamma)\sigma_1(\Omega,\Gamma) \leq \mathcal{L}^{d-1}(\Gamma)
\end{equation}
which can be considered a \emph{Dirichlet--Steklov version} of the classical Polya inequality (\cite{PolyaSzego1951}). Finally,
Equation  \eqref{eq:Ueps1} has got relevant physical meanings.
On one hand, it models the vertical displacement of a membrane under an external pressure which is concentrated near the boundary (see \cite[Theorem 4.1]{Arrieta2008} for the rigorous limit process): the considered membrane can move in the vertical direction keeping a horizontal angle.
On the other hand, solutions to \eqref{eq:Ueps1}
are stationary solutions of the heat equation that models temperature in a homogeneous and isotropic heat conductor. This is subjected to a constant heat flux through a small part of the boundary whereas the temperature is kept constant in the remaining part.

\section{Facts about \texorpdfstring{$\mathcal T_{\Omega_\e}(\e\Sigma,f)$}{the thin torsional rigidity}}\label{sec:facts}

In this section we collect some basic facts regarding the notion of \emph{thin $f$-torsional rigidity of $\e\Sigma$} introduced before.

\subsection{Basics}
Firstly, we briefly mention the variational framework for $\mathcal T_{\Omega_\e}(\e\Sigma,f)$. 
As already mentioned, by standard minimization methods, it can be proved that, for any $f\in C^1(\overline{B_{r_0\e}^+})$, there exists a unique $U^{\Omega_\e}_{\e\Sigma,f}=U^\e_{f} \in H^1_0(\Omega_\e)\setminus\{0\}$ such that 
\[
J_\e^f(U^\e_{f}) = \inf_{u\in H^1_0(\Omega_\e)} J_\e^f(u),
\]
where $J_\e^f$ is as in \eqref{eq:functional_torsion}. In particular, $U^\e_{f}$ satisfies
\begin{equation}\label{eq:weak_eq_tubes}
0=\d J_\e^f(U^\e_{f})[\varphi] = \int_{\Omega_\e}\nabla U^\e_{f}\cdot \nabla \varphi \dx-  \int_{\e\Sigma} \varphi\dfrac{\partial f}{\partial x_d}\dx' \qquad \text{for all }  \varphi \in H^1_0(\Omega_\e).  
\end{equation}
Letting $\varphi=U^\e_{f}$ in the previous equation we get
\[
	\int_{\Omega_\e}\abs{\nabla U^\e_{f}}^2\dx=\int_{\e\Sigma}U^\e_{f}\dfrac{\partial f}{\partial x_d}\dx',
\]
hence obtaining
\begin{equation}\label{eq:torsione-energia}
	\mathcal{T}_{\Omega_\e}(\e\Sigma,f)= \int_{\Omega_\e}\abs{\nabla U^\e_{f}}^2\dx = \int_{\e\Sigma}U^\e_{f}\dfrac{\partial f}{\partial x_d}\dx'.
\end{equation}

The first property deals with equivalent definitions for the \emph{thin $f$-torsional rigidity of $\e\Sigma$}.
\begin{lemma}\label{l:equivalence}
 Definition \eqref{def:torsion_functional} is equivalent to the following
\begin{equation}\label{eq:up_bound1}
	\mathcal{T}_{\Omega_\e}(\e\Sigma,f)=\sup_{u\in H^1_0(\Omega_\e)\setminus \{0\}}\frac{\left(\displaystyle\int_{\e\Sigma}u\frac{\partial
	f}{\partial x_d}\dx'\right)^2}{\displaystyle\int_{\Omega_\e}\abs{\nabla u}^2\dx}.
\end{equation}
\end{lemma}
\begin{proof}
 By definition,
\begin{align*}
	\mathcal{T}_{\Omega_\e}(\e\Sigma,f)&=\sup_{u\in H^1_0(\Omega_\e)\setminus\{0\}}\sup_{t>0}\left\{ 2t\int_{\e\Sigma}u \frac{\partial f}{\partial x_d}\dx'-t^2\int_{\Omega_\e}\abs{\nabla u}^2\dx \right\}
\end{align*}
and the inner supremum is actually attained at
\[
t=\frac{\displaystyle \int_{\e\Sigma}u\frac{\partial f}{\partial x_d}\dx'}{\displaystyle \int_{\Omega_\e}\abs{\nabla u}^2\dx}.
\]
Substituting this value into the previous equality leads to \eqref{eq:up_bound1}.
\end{proof}

\begin{remark}\label{rem:best_approx}
 We find useful to note that if $\varphi\in E(\lambda_N)$ and $u_\e$ is a perturbed eigenfunction then by \eqref{eq:str_dir_eigen} and \eqref{eq:perturbed_problem} we have
 \[
  \int_{\e\Sigma} \dfrac{\partial \varphi}{\partial x_d}u_\e = (\lambda_N - \lambda_\e) \int_\Omega \varphi u_\e. 
 \]
 From the latter equality and Lemma \ref{l:equivalence} it follows the meaningful estimate 
 \[
  \lambda_\e \mathcal T_{\Omega_\e}(\e\Sigma, \varphi) \geq \left((\lambda_N - \lambda_\e) \int_\Omega \varphi u_\e\right)^2 .
 \]
 On the other hand, if we denote the bounded linear 
 \begin{align*}
  \mathcal F: H^1_0(\Omega_\e) &\to H^{-1}(\Omega_\e)\\
  u &\mapsto -\int_{\e\Sigma} \dfrac{\partial \varphi}{\partial x_d}u, 
 \end{align*}
then by \eqref{eq:up_bound1}
\[
 \sqrt{\mathcal T_{\Omega_\e}(\e\Sigma,\varphi)}=\| \mathcal F \|_*, 
\]
as the \emph{thin $\varphi$- torsion function} $U_\varphi^\e$ is the least energy element in $\mathcal F^{-1}(\| \mathcal F \|_*) \subseteq H^1_0(\Omega_\e)$. 
\end{remark}

The next properties deal with its behavior as $\e\to0$. 
\begin{lemma}
 If $\e_1>\e_2$ then for any $f\in C^1(\overline{B_{r_0\e_1}^+})$ we have 
 \[
 \mathcal{T}_{\Omega_{\e_1}}({\e_1}\Sigma,f)\geq \mathcal{T}_{\Omega_{\e_2}}({\e_2}\Sigma,f).                                                \]
\end{lemma}
\begin{proof}
 The statement is obvious thanks to the inclusion $H^1_0(\Omega_{\e_2}) \subseteq H^1_0(\Omega_{\e_1})$.
\end{proof}

\begin{lemma}\label{l:torsione_infinitesima}
 For any $f\in E(\lambda_N)$ we have that 
 \[
  \mathcal T_{\Omega_{\e}}(\e\Sigma,f) \to 0 \quad \text{as }\e\to0.
 \]
\end{lemma}
\begin{proof}
Taking into account \eqref{eq:up_bound1}, by Cauchy-Schwarz Inequality, the trace embedding $H^1(B_{r_0}^+)\hookrightarrow L^2(\e\Sigma)$ and regularity of eigenfunctions we have that
	\begin{align}
	\mathcal{T}_\Omega(\e\Sigma,f)&\leq \sup_{u\in
		H^1_0(\Omega_\e)\setminus
		\{0\}}\dfrac{\displaystyle\int_{\e\Sigma}u^2\ds\int_{\e\Sigma}\left(\frac{\partial f}{\partial x_d}\right)^2\dx'}{\displaystyle\int_\Omega\abs{\nabla
			u}^2\dx}\notag\\\label{eq:st1}&
			=\int_{\e\Sigma}\left(\frac{\partial f}{\partial x_d}\right)^2\dx'\sup_{u\in
		H^1_0(\Omega_\e)\setminus
		\{0\}}\dfrac{\displaystyle\int_{\e\Sigma}u^2\dx'}{\displaystyle\int_\Omega\abs{\nabla
			u}^2\dx}\\
			&\label{eq:st11}\leq C_{d,\Omega,r_0}\norm{\frac{\partial f}{\partial x_d}}_{H^1(B_{r_0}^+)}^2
	\sup_{u\in H^1_0(\Omega_\e)\setminus \{0\}}\frac{\int_{\e\Sigma}u^2\dx'}{\int_{\Omega_\e}\abs{\nabla u}^2\dx}.
\end{align}
By scaling we have
	\begin{align}\label{eq:2st}
		\sup_{u\in H^1_0(\Omega_\e)\setminus
			\{0\}}\frac{\int_{\e\Sigma}u^2\dx'}{\int_{\Omega_\e}\abs{\nabla
				u}^2\dx}&\leq \sup_{u\in H^1_0(\Omega_\e)\setminus
			\{0\}}\frac{\int_{\e\Sigma}u^2\dx'}{\int_{B_{\e}^+\cup T_\e}\abs{\nabla
				u}^2\dx}\notag\\
		&\leq \sup_{u\in H^1_{0,\partial(B_1^+\cup T_1)\setminus S_1^+}(B_1^+)}\frac{\e\int_{\Sigma}u^2\dx'}{\int_{B_{1}^+\cup T_1}\abs{\nabla u}^2\dx}  =C_\Sigma\,\e
	\end{align}
	where
	\begin{equation*}
		C_\Sigma  =
		\sup_{u\in H^1_{0,\partial(B_1^+\cup T_1)\setminus S_1^+}(B_1^+)}\frac{\int_{\Sigma}u^2\dx'}{\int_{B_{1}^+\cup T_1}\abs{\nabla u}^2\dx} >0
	\end{equation*}
	and, for any compact set $K\sub\overline{B_1^+}$, the space $H^1_{0,K}(B_1^+)$ is defined as the closure of $C_c^\infty(\overline{B_1^+}\setminus K)$ with respect to the $H^1$ norm. Actually, for regular $K$ there holds $H^1_{0,K}(B_1^+)=\{u\in H^1(B_1^+)\colon u=0~\text{on }K\}$.
Invoking \eqref{eq:st11} and \eqref{eq:2st} we conclude the proof.
\end{proof}

\subsection{Blow-up analysis for the thin \texorpdfstring{$f$-torsion function}{torsion function}}
\label{subsec:mosco}
As already mentioned in the introduction, spectral stability for this problem is ensured by the results in \cite{daners2003}. It is a consequence of the uniform convergence of the resolvents. 
Nevertheless, in order to perform a blow-up analysis as  $\e\to0$ for the thin torsion function, we need a fundamental notion of convergence of sets (or functional spaces): it is the so-called \emph{convergence in the sense of Mosco}. In our setting of scaling handles it is established in \cite[Section 7]{daners2003}. We report here the definition for future reference. 
\begin{definition}\label{def:Mosco}
	Let $\mathfrak{H}_\e$, $\mathfrak{H}_0$ and $\mathfrak{H}$ be Hilbert spaces such that $\mathfrak{H}_\e,\mathfrak{H}_0\sub \mathfrak{H}$ for all $\e\in (0,1)$. We say that $\mathfrak{H}_\e$ \emph{converges to $\mathfrak{H}_0$ in the sense of Mosco in $\mathfrak{H}$} if the following hold:
	\begin{enumerate}
		\item[(M1)] if $v_\e\in \mathfrak{H}_\e$ for all $\e\in (0,1)$ and $v_\e\weak v$ weakly in $\mathfrak{H}$, as $\e\to 0$, then $v\in \mathfrak{H}_0$;
		\item[(M2)] for any $v\in \mathfrak{H}_0$ there exists a sequence $\{v_\e\}_{\e\in (0,1)}$ such that $v_\e\in \mathfrak{H}_\e$ for all $\e\in (0,1)$ and $v_\e\to v$ strongly in $\mathfrak{H}$.
	\end{enumerate}
\end{definition}

We start this last subsection giving an important lemma for the forthcoming analysis.
\begin{lemma}\label{lemma:L^2_norm} Let $f\in E(\lambda_N)$ and let $U_f^\e\in H^1_0(\Omega_\e)$ be the thin $f$-torsion function of $\e\Sigma$. Then
	\[
		\int_{\Omega_\e} |U^\e_f|^2\dx=o(\mathcal{T}_{\Omega_\e}(\e\Sigma,f)),\quad\text{as }\e\to 0.
	\]
\end{lemma}
\begin{proof}
 Let us assume by contradiction that there exists a
  sequence $\e_n\to 0$ and a constant $C>0$ such that
\begin{equation*}
  \int_{\Omega_\e}|U^{\e_n}_f|^2\dx \ge \frac1C\mathcal T_{\Omega_\e}(\Sigma_{\e_n},f).
\end{equation*}
We set
\begin{equation*}
	W_n:=\frac{U^{\e_n}_f}{\|U^{\e_n}_f\|_{L^2(\Omega_\e)}}.
\end{equation*}
We have 
\begin{equation*}
 \left\|W_n\right\|_{L^2(\Omega_\e)}=1
\end{equation*}
and recalling \eqref{eq:torsione-energia}
\begin{equation*}
 \left\|\nabla W_n\right\|_{L^2(\Omega_\e)}^2=\frac1{\|U^{\e_n}_f\|_{L^2(\Omega_\e)}^2}T_{\Omega_\e}(\Sigma_{\e_n},f)\le C.
\end{equation*}
By the weak compactness of the unit ball of $H^1_0(\Omega_{\e_0})$, the
compactness of the inclusion $H^1_0(\Omega_{\e_0})\subset L^2(\Omega_{\e_0})$ and thanks to the convergence of the perturbed domains in sense of Mosco (see Definition \ref{def:Mosco}), there
exists an increasing sequence of integers $(n_k)_{k\ge 1}$ and a
function $W\in H^1_0(\Omega)$ such that
$\left(W_{n_k}\right)_{k\ge 1}$ converges to $W$ when $k$ goes to
$+\infty$, weakly in $H^1_0(\Omega_{\e_0})$ and strongly in $L^2(\Omega_{\e_0})$.
We have that at the same time $\left\|W\right\|_{L^2(\Omega)}=1$ and $\int_\Omega \nabla W\cdot \nabla \varphi =0 $ for any $\varphi \in H^1_0(\Omega)$, therefore $W$ is identically $0$. We
have reached a contradiction and proved the lemma.
\end{proof}

We now turn to the very aim of the subsection. In order to give the blow-up result on the thin $f$-torsion function we start with an estimate on its energy. 
\begin{lemma}\label{lemma:upper_bound_torsion}
Let $f\in E(\lambda_N)$ be such that
\[
\frac{f(\e x)}{\e^k}\to \psi_k(x)\quad\text{in }C^{1,\alpha}(\overline{B_1^+})~\text{as }\e\to 0,
\]
for some integer $k\geq 1$ and some harmonic polynomial $\psi_k$, homogeneous of degree $k$. 
Then
	\[
		\mathcal{T}_{\Omega_\e}(\e\Sigma,f)=O(\e^{d+2k-2}),\quad\text{as }\e\to 0.
	\]
\end{lemma}
\begin{proof}
We start from \eqref{eq:st1}. Moreover, by assumption there holds
\begin{equation}\label{eq:st3}
	\int_{\e\Sigma}\left(\frac{\partial f}{\partial {x_d }}\right)^2\dx'=\e^{d+2k-3}\int_\Sigma \left[\frac{\partial}{\partial {x_d }}\left(\frac{f(\e x')}{\e^{k}}\right)\right]^2\dx'=O( \e^{d+2k-3}),
\end{equation}
as $\e\to 0$. The conclusion follows from
	\eqref{eq:st1},
	\eqref{eq:2st}, and \eqref{eq:st3}.
\end{proof}

\begin{theorem}\label{thm:blow_up1}
Let $U_f^\e\in H^1_0(\Omega_\e)$ be the thin $f$-torsion function of $\e\Sigma$. Under the same assumptions as in Lemma \ref{lemma:upper_bound_torsion} there holds
	\begin{equation*}
		\hat{U}_\e(x):=\dfrac{U_f^\e(\e x)}{\e^k}\to U_{\Sigma,\psi_k}^\Pi(x)\quad\text{in }\mathcal{D}^{1,2}(\Pi)~\text{as }\e\to 0,
	\end{equation*}
where $U_{\Sigma,\psi_k}^\Pi$ achieves $\mathcal{T}_\Pi(\Sigma,\psi_k)$ as defined in \eqref{eq:torsion_blow_up}.
	Moreover,
	\begin{equation*}
		\mathcal{T}_{\Omega_\e}(\e\Sigma,f)=\e^{d-2+2k}\mathcal{T}_{\Pi}(\Sigma,\psi_k)+o(\e^{d-2+2k}),\quad\text{as }\e\to 0.
	\end{equation*}
\end{theorem}
\begin{proof}
	From Lemma \ref{lemma:upper_bound_torsion} we deduce that
	\begin{equation}\label{eq:blow_up1}
		\int_{\frac{1}{\e}\Omega\cup T}|\nabla \hat{U}_\e|^2\dx\leq C,
	\end{equation}
	for some $C>0$ independent from $\e$, thus implying that
	 $\{\hat{U}_\e\}_\e$ is bounded in $\mathcal{D}^{1,2}(\R^d)$, if $\hat{U}_\e$ is meant to be trivially extended in $\R^d\setminus (\frac{1}{\e}\Omega\cup T)$. Then there exist a subsequence (still denoted by $\{\hat{U}_\e\}_\e$) and a function $W\in \mathcal{D}^{1,2}(\R^d)$ such that
	\begin{align}
		&\hat{U}_\e\rightharpoonup W \quad\text{in }\mathcal{D}^{1,2}(\R^d), \label{eq:blow_up2}\\
		&\hat{U}_\e\to W\quad\text{in }L^2(\Sigma),\label{eq:blow_up3}
	\end{align}
	as $\e\to 0$ by compactness of trace embedding. Now, for any $R>0$ such that $\Sigma\sub B_R'$, let $F_R:=B_R^+\cup T$ and $(\partial F_R)^-:=\partial F_R\cap \{x_d\leq 0\}$. Since $\hat{U}_\e\restr{F_R}\in H^1_{0,(\partial F_R)^-}(F_R)$ for $\e$ sufficiently small, one can pass to the weak limit as $\e\to 0$ and obtain that $W\in H^1_{0,(\partial F_R)^-}(F_R)$. We recall that 
	\[
		H^1_{0,(\partial F_R)^-}(F_R):=\{ u\in H^1(F_R)\colon u=0~\text{on }(\partial F_R)^- \}.
	\]
	 Combining this with the fact that $W\in \mathcal{D}^{1,2}(\R^d)$, one can easily prove that $W\in \mathcal{D}^{1,2}(\Pi)$. Now, from \eqref{eq:blow_up2} and \eqref{eq:conv_psi_gamma} and the equation satisfied by $\hat{U}_\e$, one can easily derive the equation satisfied by $W$, which is
	\[
		\int_{\Pi}\nabla W\cdot\nabla\varphi\dx=\int_{\Sigma}\varphi\frac{\partial\psi_k}{\partial x_d}\quad\text{for all }\varphi\in \mathcal{D}^{1,2}(\Pi).
	\]
	By the uniqueness of the minimizer of $\mathcal T_{\Pi}(\Sigma,\psi_k)$ (see also  \cite[Proposition 2.2]{FO} and \cite[Lemma 2.4]{FT-tubi}) we have that $W=U_{\Sigma,\psi_k}^\Pi$. Finally
	\begin{multline*}
		\e^{-d-2k+2}\mathcal{T}_{\Omega_\e}(\e\Sigma,f)=\int_{\frac{1}{\e}\Omega\cup T}|\nabla\hat{U}_\e|^2\dx=\int_{\Sigma}\hat{U}_\e\frac{\partial}{\partial x_d}\left(\frac{f(\e x)}{\e^k}\right)\dx' \\ \to \int_{\Sigma}U_{\Sigma,\psi_k}^\Pi\frac{\partial\psi_k}{\partial x_d}\dx'=\int_{\Pi}|\nabla U_{\Sigma,\psi_k}^\Pi|^2\dx=\mathcal{T}_{\Pi}(\Sigma,\psi_k)
	\end{multline*}
	as $\e\to 0$, and the proof is concluded.
\end{proof}

\section{Perturbation of eigenvalues}\label{sec:perturbation}

Our subsequent analysis aims at a perturbation theory relying on the new quantity defined above: the \emph{thin $f$-torsional rigidity}.

We introduce the quantity $\chi_\e$:
\begin{equation}
\label{eq:chiEps}
	\chi_\e^2:=\sup\{\mathcal T_{\Omega_\e}(\e\Sigma,u)\,:\,u\in E(\lambda_N)\mbox{ and } \|u\|_{L^2(\Omega)}=1\}
\end{equation}

\begin{lemma} \label{l:error}
	There holds
	\[
		\chi_\e\to 0,\quad\text{as }\e\to 0.
	\]
\end{lemma} 
\begin{proof} Let us pick $u\in E(\lambda_N)$ such that $\|u\|_{L^2(\Omega)}=1$. If $\{u_{N+i-1}\}_{i=1}^m\sub E(\lambda_N)$ denotes a $L^2(\Omega)$-orthonormal basis, we write $u=\sum_{i=1}^m c_i u_{N+i-1}$, with $\sum_{i=1}^m c_i^2=1$. Then by linearity, Cauchy-Schwarz inequality, the trivial inequality $(\sum_{i=1}^m a_i)^2\leq m\sum_{i=1}^m a_i^2$ and \eqref{eq:torsione-energia}
\begin{align*}
	\mathcal T_{\Omega_\e}(\e\Sigma,u)&=\left|\sum_{1\le i,j\le m}c_ic_j\int_{\Omega_\e} \nabla U_{u_{N+i-1}}^{\e}\cdot \nabla U_{u_{N+j-1}}^{\e}\dx\right|\\
	&\le \sum_{1\le i,j\le m}|c_i||c_j|\left(\int_{\Omega_\e} |\nabla U_{u_{N+i-1}}^{ \e }|^2\dx\right)^{\frac12}\left(\int_\Omega |\nabla U_{u_{N+j-1}}^{ \e }|^2\dx\right)^{\frac12}\\
	&=\left(\sum_{i=1}^m |c_i|\left(\int_{\Omega_\e} |\nabla U_{u_{N+i-1}}^{ \e }|^2\dx\right)^{\frac12}\right)^2\\
	&\le m\,\left(\max_{1\le i\le m}\int_{\Omega_\e} |\nabla U_{u_{N+i-1}}^{ \e }|^2\dx\right)\sum_{i=1}^m c_i^2=m\max_{1\le i\le m}\mathcal T_{\Omega_\e}(\e\Sigma,u_{N+i-1}).
\end{align*}
By Lemma \ref{l:torsione_infinitesima} $\mathcal T_{\Omega_\e}(\e\Sigma,u_{N+i-1})\to0$ for all $1\le i\le m$: the proof is complete.
\end{proof}

For $\e>0$, we denote by $\Pi_\e$ the linear mapping 
\begin{equation*}
\begin{array}{cccc}
	\Pi_\e:& E(\lambda_N)& \to     & H^1_0(\Omega_\e)\\
			 & u            & \mapsto & u+U^\e_{u},
\end{array}
\end{equation*}
where $E(\lambda_N)$ and $H^1_0(\Omega_\e)$ are considered subspaces of $L^2(\Omega)$ and $L^2(\Omega_\e)$, respectively.

\begin{lemma} \label{l:norm}
	If $M_\e:=\left\|\Pi_\e -\mathbb I\right\|_{{\mathcal L(E(\lambda_N),L^2(\Omega_\e))}}$, there holds
	 \[M_\e=o(\chi_\e),\quad\text{as }\e\to 0.\]
\end{lemma}
\begin{proof} Let $v\in E(\lambda_N)$ such that $\|v\|_{L^2(\Omega)}=1$ and let us write $v=\sum_{i=1}^m c_i u_{N+i-1}$, for some $\{c_i\}_{i=1}^m$ such that $\sum_{i=1}^m c_i^2=1$, being $\{u_{N+i-1}\}_{i=1}^m$ a basis of $E(\lambda_N)$ orthonormal in $L^2(\Omega)$. By definition, we have $(\Pi_\e-\mathbb I)v=U_v^\e$. Hence, by linearity and Cauchy-Schwarz inequality we find that
 \begin{align*}
 	\|(\Pi_\e-\mathbb{I})v\|_{\mathcal{L}(E(\lambda_N),L^2(\Omega_\e)) } &= \|U_v^\e\|_{L^2(\Omega_\e)} \\
 	&\le \sum_{i=1}^m |c_i|\|U_{u_{N+i-1}}^{ \e }\|_{L^2(\Omega_\e)} \le \left(\sum_{i=1}^m c_i^2\right)^{\frac12}\left(\sum_{i=1}^m \|U_{u_{N+i-1}}^{ \e }\|^2_{L^2(\Omega_\e)}\right)^{\frac12}\\
 	&=\left(\sum_{i=1}^m \mathcal T_{\Omega_\e}(\e\Sigma,u_{N+i-1})\frac{\|U_{u_{N+i-1}}^{ \e }\|^2_{L^2(\Omega_\e)}}{\mathcal T_{\Omega_\e}(\e\Sigma,u_{N+i-1})}\right)^{\frac12}\\ 
 	&\le \sqrt{m}\,\chi_\e\,\max_{1\le i\le m}\frac{\|U_{N+i-1}^{ \e }\|_{L^2(\Omega_\e)}}{\mathcal T_{\Omega_\e}(\e\Sigma,u_{N+i-1})^{1/2}}.
 \end{align*}
According to Lemma \ref{lemma:L^2_norm}, the last term is $o(\chi_\e)$, as $\e\to 0$, and this concludes the proof.
\end{proof}

We observe that in particular Lemma \ref{l:norm} implies that $M_\e<1$, meaning that $\Pi_\e$ is injective, for $\e$ small enough. We will always assume this to be the case in the rest of this section.

\subsection{Application of the abstract lemma}\label{sec:appLemma}

We here recall the abstract result needed in order to find good approximation of perturbed eigenvalues. It is a modification of the so-called \emph{Lemma on small eigenvalues} by Colin de Verdi\'{e}re.
\begin{proposition}[\cite{ALM2022}, Proposition 3.1]\label{p:appEV}
	Let $(\mathcal H, \|\cdot\|)$  be a Hilbert space and $q$ be a quadratic form, semi-bounded from below (not necessarily positive), with
	domain $\mathcal D$ dense in $\mathcal H$ and with discrete spectrum $\{ \nu_i \}_{i\geq1}$. Let $\{ g_i \}_{i\geq1}$ be an orthonormal basis of eigenvectors of $q$. Let $N$ and $m$ be positive integers, $F$ an $m$-dimensional subspace of $\mathcal D$ and $\{ \xi_i^F\}_{i=1}^m$ the eigenvalues of the restriction of $q$ to $F$.
	
	Assume that there exist positive constants $\gamma$ and $\delta$ such that
	\begin{itemize}
		\item[(H1)] $ 0<\delta<\gamma/\sqrt2$;
		\item[(H2)] for all $i\in\{1,\dots,m\}$, $|\nu_{N+i-1}|\le\gamma$, $\nu_{N+m}\ge \gamma$ and, if $N\ge2$, $\nu_{N-1}\le-\gamma$;
		\item[(H3)] $|q(\varphi,g)|\leq \delta\, \|\varphi \|\,\|g\|$ for all $g\in\mathcal D$ and $\varphi \in F$. 
	\end{itemize}
	Then we have
	\begin{itemize}
		\item[(i)] $\left|\nu_{N+i-1}- \xi_i^F \right|\le\frac{ 4}{\gamma}\delta^2$ for all $i=1,\ldots,m$; 
		\item[(ii)] $\left\| \Pi_N - \mathbb{I}\right\|_{\mathcal L(F,\mathcal H)} \leq { \sqrt 2}\delta/\gamma$,  where $\Pi_N$ is the projection onto the subspace of $\mathcal D$ spanned by $\{g_N,\ldots,g_{N+m-1}\}$. 
	\end{itemize}
\end{proposition}

We are going to apply Proposition \ref{p:appEV} in the following way. For $\e>0$ small enough, we introduce the following set of definitions \eqref{notfirst}--\eqref{notlast}:
\begin{align}
 &\mathcal H_\e := L^2(\Omega_\e)~\text{and }\norm{\cdot}:=\norm{\cdot}_{L^2(\Omega_\e)};\label{notfirst}\\
 &\mathcal D_\e :=H^1_0(\Omega_\e);\\
&q_\e(u):= \int_{\Omega_\e} |\nabla u|^2\dx-\lambda_N\int_{\Omega_\e} u^2\dx,\quad\text{for all }u\in\mathcal D_\e;\\
&F_{\e}:= \Pi_\e(E(\lambda_N)). \label{notlast}
\end{align}

By construction, the eigenvalues of $q_\e$ are $\{\lambda_i^\e-\lambda_N\}_{i\ge1}$. We use the notation $\nu_i^\e:=\lambda_i^\e-\lambda_N$. If $\e$ is small enough, \Cref{l:norm} implies that $\Pi_\e$ is injective, so that $\Pi_\e$ is bijective from $E(\lambda_N)$ onto $F_\e$ and $F_\e$ is proved to be $m$-dimensional. Since $\lambda_i^\e\to\lambda_i$  for all $i\in\N\setminus\{0\}$ and $\lambda_N$ is of multiplicity $m$, the assumption (H2) in Proposition \ref{p:appEV} is fulfilled for $\e>0$ small enough if we take, for instance,
\begin{equation*}
	\gamma:={\frac12}\min\{\lambda_N-\lambda_{N-1},\lambda_{N+m}-\lambda_{N+m-1}\}
\end{equation*}
when $N\ge2$, whereas $\gamma:={ \frac12}
	\left(\lambda_{2}-\lambda_{1}\right)$ when $N=1$ (in which case $m=1$).

It remains to check whether condition (H3) in Proposition \ref{p:appEV} is satisfied. Let us choose $u\in F_\e$ and $w\in \mathcal D_\e$. Since $\Pi_\e$ is injective by \Cref{l:norm}, there exists a unique $v\in E(\lambda_N)$ such that $u=\Pi_\e v$. Hence, we have
\begin{align*} 
q_\e(u,w)&=\int_{\Omega_\e}\nabla (v+U_v^\e)\cdot \nabla w\dx - \lambda_N \int_{\Omega_\e}(v+U_v^\e)w\dx\\
&= \int_{\Omega}\nabla v\cdot \nabla w \dx+ \int_{\Omega_\e}\nabla U_v^\e\cdot \nabla w \dx- \lambda_N \int_{\Omega}vw\dx - \lambda_N \int_{\Omega_\e}U_v^\e w\dx\\
&=\int_{\e\Sigma} \dfrac{\partial v}{\partial \nnu_\Omega}w\dx'
+ \int_{\Omega_\e}\nabla U_v^\e\cdot \nabla w\dx
- \lambda_N \int_{\Omega_\e}U_v^\e w\dx\\
&= -\int_{\e\Sigma} \dfrac{\partial v}{\partial x_d}w\dx' + \int_{\e\Sigma} \dfrac{\partial v}{\partial x_d}w\dx' - \lambda_N \int_{\Omega_\e}U_v^\e w\dx\\
&=- \lambda_N \int_{\Omega_\e}U_v^\e w\dx
\end{align*}
where we have used the facts that $v$ is an eigenfunction relative to $\lambda_N$, the exterior normal vector to $\Omega$ on $\e\Sigma$ is $\nnu_\Omega=(0,\dots,0,-1)$ and $U_v^\e$ satisfies \eqref{eq:weak_eq_tubes}. We then obtain
\begin{equation*}
	|q_\e(u,w)|=|q_\e(\Pi_\e v,w)|\le\lambda_N\|U_v^\e\|\|w\| \le \lambda_N M_\e \|v\|\|w\|\le\lambda_N\frac{M_\e}{1-M_\e}\|u\|\|w\|
\end{equation*}
because $\|U_v^\e\| = \|v+U_v^\e - v\|= \| (\Pi_\e - \mathbb I)v\|$ and $\|v\|=\|v - \Pi_\e v + \Pi_\e v\| \le M_\e \|v\| + \|u\|$, so that 
\[
 \|v\| \le \dfrac{\|u\|}{1-M_\e}.
\]
Lemma \ref{l:norm} then implies
\begin{equation*}
	|q_\e(u,w)|\leq \delta_\e\|u\|\|w\|,
\end{equation*}
for some $\delta_\e>0$ such that $\delta_\e=o(\chi_\e)$ as $\e\to 0$. We can now apply Proposition \ref{p:appEV} with $\delta=\delta_\e$, which implies that for $1\le i\le m$
\begin{equation}\label{eq:asympt_proof}
	\lambda_{N+i-1}^\e=\lambda_N+\xi_i^\e +o(\chi_\e^2),
\end{equation}
where $\{ \xi_i^\e \}_{i=1}^m$ are the eigenvalues of the restriction of $q_\e$ to $F_\e$.

\subsection{Analysis of the restricted quadratic form.}

We now need to study $\{ \xi_i^\e \}_{i=1}^m$. To do so,
we are now going to investigate how the quadratic form $q_\e$ acts when it is restricted to the $m$-dimensional subspace $F_\e= \Pi_\e(E(\lambda_N))$, still endowed with the $L^2(\Omega_\e)$-norm. Let us introduce the following bilinear form
\begin{equation*}				
	r_\e(u,v):= \int_{\Omega_\e}\nabla U_u^\e\cdot \nabla U_v^\e\dx+\lambda_N\int_{\Omega_\e} U_u^\e\,U_v^\e\dx,
\end{equation*}
defined for $u,v\in E(\lambda_N)$.
\begin{lemma} \label{l:restrict}
For all $\varphi_i,\varphi_j\in E(\lambda_N)$, 
\begin{equation*}
	q_\e\left(\Pi_\e \varphi_i,\Pi_\e \varphi_j\right)=-r_\e(\varphi_i,\varphi_j).
\end{equation*}
\end{lemma}
\begin{proof} 
For simplicity, in the sequel we write $U^\e_i$ in place of $U^\e_{\varphi_i}$. We have
 \begin{align*}
 	q_\e\left(\Pi_\e \varphi_i,\Pi_\e \varphi_j\right)
&=\int_{\Omega_\e}\nabla (\varphi_i+U_i^\e)\cdot \nabla(\varphi_j+U_j^\e)\dx-\lambda_N \int_{\Omega_\e} (\varphi_i+U_i^\e)\,(\varphi_j+U_j^\e)\dx\\
&=\int_{\Omega_\e}\nabla \varphi_i\cdot \nabla U_j^\e\dx + \int_{\Omega_\e}\nabla \varphi_j\cdot \nabla U_i^\e\dx + \int_{\Omega_\e}\nabla U_i^\e\cdot \nabla U_j^\e\dx\\
&\quad - \lambda_N \int_{\Omega_\e} \varphi_i\,U_j^\e \dx- \lambda_N \int_{\Omega_\e} \varphi_j\,U_i^\e \dx- \lambda_N \int_{\Omega_\e} U_i^\e\,U_j^\e\dx,
 \end{align*}
 where we have used the fact that $\varphi_i,\ \varphi_j$ are both eigenfunctions relative to $\lambda_N$. Note that the integral involving $\varphi_i$ or $\varphi_j$ taken over $\Omega_\e$ are the same if taken over $\Omega$. Integrating by parts we obtain
 \[
  \int_{\Omega}\nabla \varphi_i\cdot \nabla U_j^\e\dx = \lambda_N \int_{\Omega_\e} \varphi_i\,U_j^\e \dx+ \int_{\e\Sigma} \dfrac{\partial\varphi_i}{\partial\nnu_{\Omega}}U_j^\e\dx',
 \]
where $\nnu_\Omega=(0,\ldots,0,-1)$. 
 We can go ahead obtaining
 \begin{align*}
  q_\e\left(\Pi_\e \varphi_i,\Pi_\e \varphi_j\right)
&=- \int_{\e\Sigma}\dfrac{\partial \varphi_i}{\partial x_d}\,U_j^\e\dx' - \int_{\e\Sigma}\dfrac{\partial \varphi_j}{\partial x_d}\,U_i^\e\dx' + \int_{\Omega_\e}\nabla U_i^\e\cdot \nabla U_j^\e\dx' - \lambda_N \int_{\Omega_\e} U_i^\e\,U_j^\e\dx. 
 \end{align*}
Taking into account \eqref{eq:weak_eq_tubes} with $U_j^\e$ and $U_i^\e$ as test functions we obtain  
\[
 \int_{\Omega_\e}\nabla U_i^\e\cdot \nabla U_j^\e\dx =  \int_{\e\Sigma}\dfrac{\partial \varphi_i}{\partial x_d}\,U_j^\e \dx'=  \int_{\e\Sigma}\dfrac{\partial \varphi_j}{\partial x_d}\,U_i^\e\dx',
\]
respectively.
In this way,
\begin{align*}
 q_\e\left(\Pi_\e \varphi_i,\Pi_\e \varphi_j\right)
&= - \int_{\Omega_\e}\nabla U_i^\e\cdot \nabla U_j^\e\dx - \lambda_N \int_{\Omega_\e} U_i^\e\,U_j^\e\dx', 
\end{align*}
and the proof is concluded. 
\end{proof}

\begin{remark}\label{r:eige}
Note that $\{\Pi_\e\varphi_i\}_{i=1,\ldots,m}$ is a basis of $\Pi_\e(E(\lambda_N))$, but $\Pi_\e\varphi_i$ are not orthogonal to each other. This is true only at the limit as $\e\to0$, since $\Pi_\e\varphi_i \to \varphi_i$ as $\e\to0$ for any $i=1,\ldots,m$. Thus, Lemma \ref{lemma:L^2_norm} and Lemma \ref{l:restrict} imply that
 \[
  \xi_j^\e = \mu_j^\e + o({\chi_\e}^2) \qquad \text{as }\e\to0,
 \]
where $\mu_j^\e$ denote the eigenvalues of the form $r_\e(\cdot,\cdot)$ defined on $ E(\lambda_N)$. 
\end{remark}

Therefore, in view of \eqref{eq:asympt_proof} and \Cref{r:eige}, the proof of \Cref{thm:approxEVs} is complete.

\section{Ramification of eigenvalues}\label{sec:ramification}

The aim of the present section is to prove \Cref{thm:orderEVs}. In order to investigate the occurrence of ramification of multiple eigenvalues, we need to study $\mu_j^\e$, i.e. the eigenvalues of the form $r_\e(\cdot,\cdot)$ defined on $ E(\lambda_N)$. As already mentioned in the introduction, we expect that as $\e>0$ multiple eigenvalues split according to the order of vanishing of suitably chosen limit eigenfunctions at the origin. Hence, we first introduce 
the aforementioned \textit{order decomposition} of $E(\lambda_N)$, which drives us towards the choice of the proper limit eigenbasis. Secondly, we iteratively apply the abstract result \Cref{p:appEV}, by choosing smaller and smaller approximating spaces $F$.

\subsection{Order decomposition of the eigenspace}

For clarity of exposition, we report here the statement of Proposition \ref{prop:DecompESintro}, which we are going to prove. 
\begin{proposition}\label{prop:DecompES} There exists a decomposition of $E(\lambda_N)$ into a sum of orthogonal subspaces
\[E(\lambda_N)=E_1\oplus\dots\oplus E_p,\] 
for some integer $p\geq 1$, and an associated finite increasing sequence of integers
\[0<k_1<\dots<k_p\]
such that, for all $1\le j \le p$, a function in $E_j\setminus\{0\}$ has the order of vanishing $k_j$ at $0$, that is
\[
	\frac{\varphi(rx)}{r^{k_j}}\to \psi_{k_j}(x)\quad\text{in }C^{1,\alpha}(\overline{B_1^+}),~\text{as }\e\to 0,
\]
for some harmonic polynomial $\psi_{k_j}$, homogeneous of degree $k_j$ and odd with respect to $x_d$.
 In addition, such a decomposition is unique. 
\end{proposition}

\begin{proof} 
Given any $u\in E(\lambda_N)$, let us consider its restriction to $B_{R}^+=\{x\in B_{R}\colon x_d>0\}$, for $R<r_0$ sufficiently small so that $B_R^+\sub\Omega$ (with $r_0$ as in \eqref{eq:ass_flat}), let us extend it to $B_{R}$ oddly with respect to $x_d$, and let us call it $\tilde{u}$. Then
\[
	-\Delta \tilde{u}=\lambda_N \tilde{u},\quad\text{in }B_{R}.
\]
 Hence, in view of classical regularity results, $\tilde{u}$ is analytic at $0$ and any truncation of its Taylor expansion at $0$ is odd with respect to $x_d$. 
For any $k\in \N$, let us define the mapping $\Pi_k:E(\lambda_N)\to \R_k^{\mathrm{odd}}[X_1,\dots,X_d]$ that associates to a function its (upper) Taylor expansion at $ 0$, truncated to order $k$. Here $\R_k^{\mathrm{odd}}[X_1,\dots,X_d]$ is the the space of polynomials odd with respect to $x_d$ of degree at most $k$.
The proof can then proceed as in \cite[Appendix A]{ALM2022}. 

 \end{proof}

\begin{remark}\label{rem:MaxDim} Let $E(\lambda_N)=E_1\oplus\dots\oplus E_p$, be the order decomposition of Proposition \ref{prop:DecompES}. Then the dimension of $E_j$ is at most the dimension of the space of spherical harmonics in $d$ variables of degree $k_j$ (see, {\it e.g.}, \cite[pp. 159--165]{Berger1971}) vanishing on $\{x_d=0\}$. Explicitly,
\[\mbox{dim}(E_{j})\le \binom{k_j + {
 d} - 2}{k_j} + \binom{k_j + {d} - 3}{k_j-1} - 1.
\] 
As a consequence, 
in the case $d=2$ $\mbox{dim}(E_{j})\le 1$ for all $1\le j\le p$.
\end{remark}

\subsection{\texorpdfstring{Eigenvalues $\mu_j^\e$}{Approximating eigenvalues}}

Let us denote by
\[E(\lambda_N)=E_1\oplus\dots\oplus E_p\] 
the order decomposition of the eigenspace $E(\lambda_N)$ (see Proposition \ref{prop:DecompES}),  with
\[0<k_1<\dots<k_p\]
the associated finite sequence of vanishing orders. 
We introduce the bilinear forms
\begin{equation*}
 \mathfrak T_\ell(u,v)=\int_{\Pi} \nabla U_{\Sigma,u}^\Pi\cdot \nabla U_{\Sigma,v}^\Pi \ \qquad \text{for }u,v\in E_\ell\subseteq E(\lambda_N)
\end{equation*}
where $U_{\Sigma,u}^\Pi$ and $U_{\Sigma,v}^\Pi$ achieve 
\[
	\mathcal{T}_\Pi(\Sigma,\mathfrak{B}_\ell u)=-2\inf_{w\in \mathcal D^{1,2}(\Pi)} \left\{ \frac12 \int_\Pi |\nabla w|^2\dx -\int_{\Sigma} w\dfrac{\partial \mathfrak{B}_\ell u}{\partial x_d}\dx' \right\}
\]
and
\[
\mathcal{T}_\Pi(\Sigma,\mathfrak{B}_\ell v)=-2\inf_{w\in \mathcal D^{1,2}(\Pi)} \left\{ \frac12 \int_\Pi |\nabla w|^2\dx -\int_{\Sigma} w\dfrac{\partial \mathfrak{B}_\ell v}{\partial x_d}\dx' \right\}
\]
respectively. 
For simplicity of exposition, we assume that the orthonormal basis $\{u_{N+i-1}\}_{i=1}^m$ agrees with the order decomposition and diagonalizes each of the quadratic forms $\mathfrak T_\ell$. Explicitly, this means that, for all $\ell\in\{1,\dots,p\}$, 
\begin{equation*}
	E_\ell=\mbox{span}\{u_{N+m_1+\dots+m_{\ell-1}},\dots,u_{N+m_1+\dots+m_{\ell-1}+m_\ell-1}\}
\end{equation*}
and, for all $1\le s<t\le  m_\ell$,
\begin{equation*}
\mathfrak T_\ell(u_{N+m_1+\dots+m_{\ell-1}+s-1},u_{N+m_1+\dots+m_{\ell-1}+t-1})=0.
\end{equation*}
It follows that, for all $1\le s\le  m_\ell$,
\begin{equation*}
\mathfrak T_\ell(u_{N+m_1+\dots+m_{\ell-1}+s-1},u_{N+m_1+\dots+m_{\ell-1}+s-1})=\mu_{\ell,s}.
\end{equation*}

According to Remark \ref{r:eige}, we start from the lowest rate of convergence $k_1$ and we look for the $m_1$ largest eigenvalues (as $\e\to0$) of the matrix of the quadratic form $r_\e$ in the basis $\{u_{N+i-1}\}_{i=1}^m$, namely $A_\e$. 
It follows from Lemma \ref{lemma:L^2_norm} and Theorem \ref{thm:blow_up1} that 
\begin{equation*}
A_\e=
\left(
\begin{array}{cccccc}
	& &&						&       &						   \\
	&\mathbf 0&&      					&\mathbf 0      &						   \\
	& &&						& 		&						   \\
	& &&\mu_{1,1}\,\e^{d-2+2k_1}&	    &			\mathbf 0		   \\
	&\mathbf 0&&					&\ddots &					   \\
	& &&		\mathbf 0			& 		&\mu_{1,m_1}\,\e^{d-2+2k_1}
\end{array}
\right)
+o\left(\e^{d-2+2k_1}\right).
\end{equation*}  
Using the min-max characterization of eigenvalues, we conclude that, for $1\le i\le m_1$,
\begin{equation*}
	\mu^\e_{i}=\mu_{1,i}\,\e^{d-2+2k_1}+o\left(\e^{d-2+2k_1}\right)
\end{equation*}
and, for $m_1+1\le i\le m$,
\begin{equation*}
	\mu^\e_{i}=o\left(\e^{d-2+2k_1}\right).
\end{equation*}
Proposition \ref{p:appEV}, Remark \ref{r:eige} and the fact that $\chi_\e^2$ and $\e^{d-2+2k_1}$ are of the same order, tell us that the same estimates hold for the difference $\nu^\e_{N-1+i}:=\lambda_{N-1+i}^\e-\lambda_N$:  for $1\le i\le m_1$
\begin{equation*}
	\nu^\e_{N-1+i}=-
	\mu_{1,i}\e^{d-2+2k_1}+o\left(\e^{d-2+2k_1}\right)
\end{equation*}
and, for $m_1+1\le i\le m$,
\begin{equation*}
	\nu^\e_{N-1+i}=o\left(\e^{d-2+2k_1}\right).
\end{equation*}

The rest of the proof consists of a step-by-step procedure, in which we rescale the quadratic form $q_\e$ and apply the same arguments in order to identify successive groups of eigenvalues converging to $\lambda_N$ with the same rate. Let us sketch the next step. We set, for $u,v\in\mathcal D_\e$,
\begin{equation*}
	q_{2}^\e(u,v)\equiv\frac1{\e^{d-2+2k_1}}q_\e(u,v),
\end{equation*}
and we define the subspace
\begin{equation*}
	F_{2}^\e=\Pi_\e(E_2\oplus\dots\oplus E_{p}).
\end{equation*}

The eigenvalues of $q^\e_{2}$ are $\left\{\frac1{\e^{d-2+2k_1}}\nu_i^\e\right\}_{i\ge1}=\frac1{\e^{d-2+2k_1}}\left\{\lambda_i^\e - \lambda_N\right\}_{i\ge1}$. 
We know from the first step that, for $1\le i\le m_1$,
\begin{equation*}
	\lim_{\e\to0}\frac1{\e^{d-2+2k_1}}\nu^\e_{N-1+i}=-\mu_{1,i}<0.
\end{equation*}
It follows immediately that there exists $\gamma>0$ such that, for $\e>0$ small enough,
\begin{equation*}
  \left|\frac1{\e^{d-2+2k_1}}\nu_{N-1+i}^\e\right|\le\gamma \mbox{ for } m_1+1\le i\le m
  \quad \text{and}\quad
	\frac1{\e^{d-2+2k_1}}\nu_{N+m}^\e\ge2\gamma;
\end{equation*}
whereas in case $N\ge2$ even
\begin{equation*}
	\frac1{\e^{d-2+2k_1}}\nu_{N-1}^{\e}\le -2\gamma.
\end{equation*}
Repeating the arguments of Section \ref{sec:appLemma}, we can show that for all $v\in F_{2}^\e$ and $w\in \mathcal D_\e$,
\begin{equation*}
	\left|q^\e_{2}(v,w)\right|\le o\left(\left(\frac{\e^{d-2+2k_{2}}}{\e^{d-2+2k_1}}\right)^{1/2}\right) \|v\|\|w\|
	=o\left( \e^{k_{2} - k_{1}} \right) \|v\|\|w\|.
\end{equation*}
Using the arguments in the proof of Theorem \ref{thm:approxEVs} and in the first step, we conclude that, for $1+m_1\le i\le m_1+m_2$,
\begin{equation*}
	\frac1{\e^{d-2+2k_1}}\nu^\e_{N-1+i}=-\mu_{2,i-m_1}\,\e^{2k_2 - 2k_1}+o\left(\e^{2k_2 - 2k_1}\right)
\end{equation*}
and, for $m_1+m_2+1\le i\le m$,
\begin{equation*}
	\frac1{\e^{d-2+2k_1}}\nu^\e_{N-1+i}=o\left(\frac{\e^{d-2+2k_2}}{\e^{d-2+2k_1}}\right).
\end{equation*}
This gives us finally, for $m_1+1\le i\le m_1+ m_2$,
\begin{equation*}
	\nu^\e_{N-1+i}=-\mu_{2,i-m_1}\,\e^{d-2+2k_2}+o\left(\e^{d-2+2k_2}\right)
\end{equation*}
and, for $m_1+m_2+1\le i\le m$,
\begin{equation*}
	\nu^\e_{N-1+i}=o\left(\e^{d-2+2k_2}\right).
\end{equation*}
Carrying on the procedure for $\ell$ from $3$ to $m$, we reach the conclusion.

\begin{remark}
By Remark \ref{rem:MaxDim}, in dimension $d=2$ the eigenfunctions associated to a multiple eigenvalue have necessarily different vanishing order at $0$. We then recover the result proved in \cite[Section 11]{Gadylshin2005tubi}.
\end{remark}

%
%

\section*{Acknowledgments}
L. Abatangelo is supported by the Italian Ministery MUR grant Dipartimento di Eccellenza 2023-2027.
 R. Ognibene is partially supported by the project ERC VAREG - \emph{Variational approach to the regularity of the free boundaries} (grant agreement No. 853404) and by the INdAM-GNAMPA 2022 project \emph{Questioni di esistenza e unicità per problemi nonlocali con potenziali di tipo Hardy}. Part of this work was developed while R. Ognibene was in residence at Institut Mittag-Leffler in Djursholm, Stockholm (Sweden) during the semester \emph{Geometric Aspects of Nonlinear Partial Differential Equations} in 2022, supported by the Swedish Research Council under grant no. 2016-06596.
 

\bibliography{biblio}

\begin{thebibliography}{10}

\bibitem{AFL2018}
{\sc Abatangelo, L., Felli, V., and L\'{e}na, C.}
\newblock Eigenvalue variation under moving mixed {D}irichlet-{N}eumann
  boundary conditions and applications.
\newblock {\em ESAIM Control Optim. Calc. Var. 26\/} (2020), Paper No. 39, 47.

\bibitem{AFT2014}
{\sc Abatangelo, L., Felli, V., and Terracini, S.}
\newblock On the sharp effect of attaching a thin handle on the spectral rate
  of convergence.
\newblock {\em J. Funct. Anal. 266}, 6 (2014), 3632--3684.

\bibitem{ALM2022}
{\sc Abatangelo, L., L\'{e}na, C., and Musolino, P.}
\newblock Ramification of multiple eigenvalues for the {D}irichlet-{L}aplacian
  in perforated domains.
\newblock {\em J. Funct. Anal. 283}, 12 (2022), Paper No. 109718.

\bibitem{Agranovich2006}
{\sc Agranovich, M.~S.}
\newblock On a mixed {P}oincar\'{e}-{S}teklov type spectral problem in a
  {L}ipschitz domain.
\newblock {\em Russ. J. Math. Phys. 13}, 3 (2006), 239--244.

\bibitem{Arrieta2008}
{\sc Arrieta, J.~M., Jim\'{e}nez-Casas, A., and Rodr\'{\i}guez-Bernal, A.}
\newblock Flux terms and {R}obin boundary conditions as limit of reactions and
  potentials concentrating at the boundary.
\newblock {\em Rev. Mat. Iberoam. 24}, 1 (2008), 183--211.

\bibitem{Berger1971}
{\sc Berger, M., Gauduchon, P., and Mazet, E.}
\newblock {\em Le spectre d'une vari\'{e}t\'{e} riemannienne}.
\newblock Lecture Notes in Mathematics, Vol. 194. Springer-Verlag, Berlin-New
  York, 1971.

\bibitem{Bers1955}
{\sc Bers, L.}
\newblock {Local behavior of solutions of general linear elliptic equations}.
\newblock {\em Communications on Pure and Applied Mathematics 8}, 4 (nov 1955),
  473--496.

\bibitem{BrascoGonzalezIspizua2022}
{\sc Brasco, L., Gonzalez, M., and Ispizua, M.}
\newblock A steklov version of the torsional rigidity, Preprint 2022, {\tt
  https://arxiv.org/abs/2207.04816v1}.

\bibitem{ColindeV1986}
{\sc Colin~de Verdi\`ere, Y.}
\newblock Sur la multiplicit\'{e} de la premi\`ere valeur propre non nulle du
  laplacien.
\newblock {\em Comment. Math. Helv. 61}, 2 (1986), 254--270.

\bibitem{CollinsTaylor2018}
{\sc Collins, C.~D., and Taylor, J.~L.}
\newblock Eigenvalue convergence on perturbed {L}ipschitz domains for elliptic
  systems with mixed general decompositions of the boundary.
\newblock {\em J. Differential Equations 265}, 12 (2018), 6187--6209.

\bibitem{Courtois1995}
{\sc Courtois, G.}
\newblock Spectrum of manifolds with holes.
\newblock {\em J. Funct. Anal. 134}, 1 (1995), 194--221.

\bibitem{daners2003}
{\sc Daners, D.}
\newblock Dirichlet problems on varying domains.
\newblock {\em J. Differential Equations 188}, 2 (2003), 591--624.

\bibitem{FFT2011}
{\sc Felli, V., Ferrero, A., and Terracini, S.}
\newblock Asymptotic behavior of solutions to {S}chr\"{o}dinger equations near
  an isolated singularity of the electromagnetic potential.
\newblock {\em J. Eur. Math. Soc. (JEMS) 13}, 1 (2011), 119--174.

\bibitem{FNO1}
{\sc Felli, V., Noris, B., and Ognibene, R.}
\newblock Eigenvalues of the {L}aplacian with moving mixed boundary conditions:
  the case of disappearing {D}irichlet region.
\newblock {\em Calc. Var. Partial Differential Equations 60}, 1 (2021), Paper
  No. 12, 33.

\bibitem{FNO2}
{\sc Felli, V., Noris, B., and Ognibene, R.}
\newblock Eigenvalues of the {L}aplacian with moving mixed boundary conditions:
  the case of disappearing {N}eumann region.
\newblock {\em J. Differential Equations 320\/} (2022), 247--315.

\bibitem{FO}
{\sc Felli, V., and Ognibene, R.}
\newblock Sharp convergence rate of eigenvalues in a domain with a shrinking
  tube.
\newblock {\em J. Differential Equations 269}, 1 (2020), 713--763.

\bibitem{FT-tubi}
{\sc Felli, V., and Terracini, S.}
\newblock Singularity of eigenfunctions at the junction of shrinking tubes,
  {P}art {I}.
\newblock {\em J. Differential Equations 255}, 4 (2013), 633--700.

\bibitem{Gadyl'shin1992}
{\sc Gadyl'shin, R.}
\newblock Ramification of a multiple eigenvalue of the {D}irichlet problem for
  the {L}aplacian under singular perturbation of the boundary condition.
\newblock {\em Mathematical Notes 52}, 4 (1992), 1020--1029.

\bibitem{Gadylshin2005tubi}
{\sc Gadyl'shin, R.~R.}
\newblock The method of matching asymptotic expansions in a singularly
  perturbed boundary value problem for the {L}aplace operator.
\newblock {\em Sovrem. Mat. Prilozh.}, 5, Asimptot. Metody Funkts. Anal.
  (2003), 3--32.

\bibitem{HHH99}
{\sc Hardt, R., Hoffmann-Ostenhof, M., Hoffmann-Ostenhof, T., and Nadirashvili,
  N.}
\newblock Critical sets of solutions to elliptic equations.
\newblock {\em J. Differential Geom. 51}, 2 (1999), 359--373.

\bibitem{henrot2018}
{\sc Henrot, A., and Pierre, M.}
\newblock {\em Shape variation and optimization. {A} geometrical analysis},
  vol.~28 of {\em EMS Tracts Math.}
\newblock Z{\"u}rich: European Mathematical Society (EMS), 2018.

\bibitem{PolyaSzego1951}
{\sc P\'{o}lya, G., and Szeg\"{o}, G.}
\newblock {\em Isoperimetric {I}nequalities in {M}athematical {P}hysics}.
\newblock Annals of Mathematics Studies, No. 27. Princeton University Press,
  Princeton, N. J., 1951.

\bibitem{Seo2021}
{\sc Seo, D.-H.}
\newblock A shape optimization problem for the first mixed
  {S}teklov-{D}irichlet eigenvalue.
\newblock {\em Ann. Global Anal. Geom. 59}, 3 (2021), 345--365.

\bibitem{Taylor2013}
{\sc Taylor, J.~L.}
\newblock Convergence of {D}irichlet eigenvalues for elliptic systems on
  perturbed domains.
\newblock {\em J. Spectr. Theory 3}, 3 (2013), 293--316.

\end{thebibliography}
\bibliographystyle{acm} 

\end{document}